\documentclass[reqno]{amsart}
\usepackage{amsmath}
\usepackage{amsfonts}
\usepackage{mathrsfs}
\usepackage{amssymb}
\usepackage{amsthm}
\usepackage{bbm}
\usepackage{bm}
\usepackage{ dsfont }
\usepackage{appendix}
\usepackage[normalem]{ulem}
\usepackage{enumerate}
\numberwithin{equation}{section}
\usepackage{enumitem}

\usepackage{mathtools}

\usepackage{xcolor}
\definecolor{linkcolour}{rgb}{0.15,0.15,0.55}

\usepackage[colorlinks]{hyperref}
	\hypersetup{
		colorlinks = true,
		linkcolor = blue}

\newcommand{\vertiii}[1]{{\left\vert\kern-0.25ex\left\vert\kern-0.25ex\left\vert #1 
    \right\vert\kern-0.25ex\right\vert\kern-0.25ex\right\vert}}

\newcommand*{\qtext}[1]{\quad\text{#1}\quad}
\newcommand*{\qqtext}[1]{\qquad\text{#1}\qquad}

\newcommand*{\coleq}{\mathrel{\mathop:}=}
\newcommand*{\eqcol}{=\mathrel{\mathop:}}

\newcommand{\bpf}[1][Proof]{{\noindent {\sc #1: }}}
\newcommand{\epf}{{{\hfill $\Box$ \smallskip}}}

\newcommand{\R}{\mathbb{R}}
\newcommand{\N}{\mathbb{N}}

\newcommand{\Z}{\mathbb{Z}}

\newcommand{\tbf}{\mathbf{t}}

\newcommand{\ibf}{\mathbf{i}}

\newcommand{\xbf}{\mathbf{x}} 

\newcommand{\Pp}{\mathbf{P}}

\newcommand{\E}{\mathbf{E}}

\newcommand{\Bc}{\mathcal{B}}

\newcommand{\AC}{\mathcal{A}}

\newcommand{\Cc}{\mathcal{C}}

\newcommand{\F}{\mathcal{F}}

\newcommand{\eps}{\epsilon}

\newcommand{\id}{\mathbbm{1}}

\newcommand{\Real}{\text{Re}}

\newcommand{\zbf}{\mathbf{z}}

\newcommand{\Psf}{\mathbf{P}}

\newcommand\Item[1][]{%
  \ifx\relax#1\relax  \item \else \item[#1] \fi
  \abovedisplayskip=0pt\abovedisplayshortskip=0pt~\vspace*{-\baselineskip}}

\newcommand{\sublin}{\sigma}

\newtheorem{theorem}{Theorem}[section]
\newtheorem{lemma}{Lemma}[section]
\newtheorem{claim}{Claim}[section]
\newtheorem{remark}{Remark}[section]

\DeclareFontFamily{U}{mathc}{}
\DeclareFontShape{U}{mathc}{m}{it}%
{<->s*[1.03] mathc10}{}

\DeclareMathAlphabet{\mathscr}{U}{mathc}{m}{it}

\title[Factorization formula]{On a Factorization Formula for the Partition Function of Directed Polymers.}

\author[T. Hurth, K. Khanin, B. Navarro Lameda, F. Nazarov]{Tobias Hurth, Konstantin Khanin, Beatriz Navarro Lameda,
Fedor Nazarov} 
\address{Universit\'e de Neuch\^atel, Institut de math\'ematiques, Rue Emile-Argand 11, CH-2000 Neuch\^atel}
\address{Department of Mathematics, University of Toronto, Bahen Centre, 40 St. George St., Toronto, ON M5S 2E4, Canada}
\address{Department of Mathematics, University College London, Gower Street, London WC1E 6BT, UK}
\address{Department of Mathematical Sciences, Kent State University, 800 E. Summit Street, Kent OH 44242, USA}

\makeatletter
\@namedef{r@tocindent4}{0pt}
\@namedef{r@tocindent5}{0pt}
\makeatother

\makeatletter

\renewcommand\section{\@startsection {section}{1}{\z@}%
                                   {-3.5ex \@plus -1ex \@minus -.2ex}%
                                   {2.3ex \@plus.2ex}%
                                   {\centering\normalfont\Large\scshape}}% from \Large
\renewcommand\subsection{\@startsection{subsection}{2}{\z@}%
                                     {-3.25ex\@plus -1ex \@minus -.2ex}%
                                     {1.5ex \@plus .2ex}%
                                     {\normalfont\bfseries}}% from \large
\renewcommand\subsubsection{\@startsection{subsubsection}{3}{\z@}%
                                     {-3.25ex\@plus -1ex \@minus -.2ex}%
                                     {-1ex \@plus -3ex}%
                                     {\normalfont\bfseries}}% from \normalsize

\makeatother

\makeatletter
\renewenvironment{proof}[1][\proofname]{\par
  \pushQED{\qed}%
  \normalfont \topsep6\p@\@plus6\p@\relax
  \trivlist
  \item[\hskip\labelsep
        \bfseries
    #1\@addpunct{.}]\ignorespaces% punctuation
}{%
  \popQED\endtrivlist\@endpefalse
}
\makeatother

\makeatletter
\def\@tocline#1#2#3#4#5#6#7{\relax
  \ifnum #1>\c@tocdepth % then omit
  \else
    \par \addpenalty\@secpenalty\addvspace{#2}%
    \begingroup \hyphenpenalty\@M
    \@ifempty{#4}{%
      \@tempdima\csname r@tocindent\number#1\endcsname\relax
    }{%
      \@tempdima#4\relax
    }%
    \parindent\z@ \leftskip#3\relax \advance\leftskip\@tempdima\relax
    \rightskip\@pnumwidth plus4em \parfillskip-\@pnumwidth
    #5\leavevmode\hskip-\@tempdima
      \ifcase #1
       \or\or \hskip 1.8em \or \hskip 4.5em \else \hskip 3em \fi%
      #6\nobreak\relax
    \hfill\hbox to\@pnumwidth{\@tocpagenum{#7}}\par% <---- \dotfill -> \hfill
    \nobreak
    \endgroup
  \fi}
\makeatother

\begin{document}

\begin{abstract}

We prove a factorization formula for the point-to-point partition function associated with a model of directed polymers on the space-time lattice $\Z^{d+1}$, subject to an i.i.d. random potential and in the regime of weak disorder. In particular, we show that the error term in the factorization formula is uniformly small for starting and end points $x, y$ in the sub-ballistic regime $\| x - y \| \leq t^{\sigma}$, where $\sigma < 1$ can be arbitrarily close to $1$. This extends a result obtained in~\cite{Sinai_95}. We also derive asymptotics for spatial and temporal correlations of the field of limiting partition functions.

%We consider a dynamical system characterized by random switching between two linear vector fields $u_i(x_1, x_2) = A (x_1 - i, x_2 - i)^{\top}$, $i \in \{0,1\}$, where $A$ is a $(2 \times 2)$ diagonal matrix with distinct negative eigenvalues. For most choices of switching rates, we identify the set of points where the invariant density has a singularity. We also present two switching systems, motivated by applications in biology, that can be reduced to the present one. 
\end{abstract}

\maketitle

{\bf AMS classifiers:}   60H15, 35R60, 37L40, 60K35, 60F05 

%\tableofcontents

%\newpage

%======================================
%======================================
\section{{Introduction}}      \label{sec:intro} 

The theory of directed polymers has been actively studied in the mathematical and physical literature in the last 30 years.
From the point of view of probability theory and statistical mechanics, directed polymers are random walks in a random
potential. The probability distribution for a random path $\gamma$ of length $t$ is given by the Gibbs distribution
$P^t_\omega(\gamma)= \frac{1}{Z^t_\omega}\exp{\left[-\beta H^t_\omega(\gamma)\right]}$, where $\beta$ is the inverse temperature, $H^t_\omega(\gamma)$ is the total energy of the interaction
between the path $\gamma$ and a fixed realization of the external random potential, and the normalizing factor $Z^t_\omega$ is the partition function. The random potential is a functional defined on some probability space, and a point $\omega$ in this probability space 
completely characterizes a fixed realization of the potential. In this paper we are interested only in the case of non-stationary
time-dependent random potentials. The simplest setting corresponds to the discrete space-time lattice $\Z^{d+1}$, where $d$ is the spatial
dimension. In this case the random potential normally is assumed to be given by the i.i.d. field $\omega=\{\xi(x,i): \, x\in \Z^d, i\in \Z\}$,
and $H_\omega^t= -\sum_{i=0}^t {\xi(\gamma_i,i)}$. As usual one is interested in the asymptotic behavior of directed polymers as $t\to \infty$.

\

The first rigorous results for directed polymers were obtained by Imbrie and Spencer (\cite{Imbrie1988}), Bolthausen (\cite{Bolthausen}), and Sinai(\cite{Sinai_95}).
It was proved that in the case of weak disorder, namely when $d\geq 3$ and $|\beta |$ is small, the polymer almost surely has diffusive behavior
with a non-random covariance matrix. It was later proved by P. Carmona and Hu (\cite{CarmonaHu}), and Comets, Shiga, and Yoshida (\cite{Comets_Shiga_Yoshida}) that in the cases
$d=1,2$, and $d\geq 3$ with $|\beta |$ large, the asymptotic behavior is very different. In this regime, called strong disorder, the directed polymers
are not spreading as $t\to \infty$ but remain concentrated in certain random places.

\

Sinai's approach in~\cite{Sinai_95} is based on the study of asymptotic properties of partition functions $Z^t_\omega$ as $t \to \infty$. It turns out that
if the polymer starts at a point $x$ at time $s$, then in the limit $t\to \infty$ the properly normalized partition function converges almost surely
to a random variable $Z^\infty_{x,s}$. Here, in order to simplify notation, we are not indicating the dependence on $\omega$. In a similar way one can consider backward in time partition functions, and prove that after
the same normalization they also converge to limiting partition functions $Z_{-\infty}^{y,t}$, where $(y,t)$ is the endpoint of the polymer. 
The proof of the diffusive behavior follows from a factorization formula proved by Sinai. Namely, a bridging partition function
$Z_{x,s}^{y,t}$ corresponding to the random-walk bridge between points $(x,s), \, (y,t), \, t>s,$ satisfies the following asymptotic relation:
\begin{equation}
\label{FF}
Z_{x,s}^{y,t}= q^{y-x}_{t-s}(Z_{x,s}^{\infty} Z_{-\infty}^{y,t} +\delta_{x,s}^{y,t}),
\end{equation}
where $q^{y-x}_{t-s}$ is the transition probability of the simple symmetric random walk, and a small error term $\delta_{x,s}^{y,t}$ tends to 
zero as $t-s \to \infty$, provided $y-x$ belongs to the diffusive region: $\|y-x\|=O(\sqrt{t-s})$. Later, Sinai's formula was extended
by Kifer (\cite{Kifer}) to the continuous setting.

\

The interest in the asymptotic behavior of directed polymers is largely motivated by the connection between directed polymers and
the theory of the stochastic heat equation

$$\partial_tZ(x,t)=\frac{1}{2} \Delta Z(x,t) + \xi^\omega(x,t)Z(x,t)$$

and the random Hamilton-Jacobi equation

$$\partial_t \Phi (x,t) + \frac{1}{2}|\nabla \Phi (x,t)|^2 = \frac{1}{2}\Delta \Phi(x,t) - \xi^\omega(x,t),$$
which is related to the stochastic heat equation through the Hopf-Cole transformation $\Phi(x,t) = - \ln Z(x,t)$.
The connection between directed polymers and the stochastic heat equation is a direct consequence of the Feynman-Kac formula (\cite{BKh}).

\

The main conjecture about the asymptotic behavior of the solutions to the random Hamilton-Jacobi equation can be formulated in 
the following way. For a fixed value of the average velocity $b=\langle \nabla \Phi (x, \cdot)\rangle$, which is preserved by the equation,
with probability one there exists a unique (up to an additive constant) global solution. This means that solutions starting from 
two different initial conditions $\Phi_1(x,0)=b\cdot x +\Psi_1(x,0), \, \Phi_2(x,0)=b\cdot x +\Psi_2(x,0)$ approach each other up to an additive constant as $t\to \infty$, provided $\Psi_1(x,0)$ and $\Psi_2(x,0)$ are functions of sublinear growth in $\|x\|$ (\cite{BKh}).

\

In terms of the stochastic heat equation, a similar uniqueness statement up to a multiplicative constant conjecturally holds for two
initial conditions of the form $Z_1(x,0)=\exp{[-b\cdot x -\Psi_1(x,0)]}$ and  $Z_2(x,0)=\exp{[-b\cdot x -\Psi_2(x,0)]}$. 

\

In order to be able to prove the above conjecture in the weak-disorder case one has to extend the factorization formula (\ref{FF})
to a much larger scale. This is the purpose of the present article: We prove that the factorization formula holds for  
$\|x-y\|<(t-s)^\sigma$, where $\sigma$ can be taken arbitrarily close to 1. Compared to~\cite{Sinai_95}, such an extension of
the factorization formula requires very different analytical methods.

\

In this paper we restrict ourselves to the simplest discrete case, i.e. polymers live on the discrete space-time lattice $\Z^{d+1}$ and the potential is induced by an i.i.d. field of random variables. This allows us to make the exposition more transparent.
However, in order to prove the uniqueness conjecture for the stochastic heat equation in the weak-disorder regime, one needs to consider the parabolic Anderson model, which is discrete in space and continuous in time. The proof of the factorization formula in this semi-discrete setting,
which is based on similar ideas but technically more involved, will be published elsewhere. The proof of the uniqueness conjecture itself will be published separately as well. 

\

We conclude the introduction with several remarks:

\

1. The factorization formula can be extended to the full sub-ballistic regime $\|y-x\|=o(t-s)$. We are considering a smaller region
$\|x-y\|<(t-s)^\sigma$ which allows for effective estimates of the smallness of the error term $\delta_{x,s}^{y,t}$.

\

2. We believe that a similar factorization formula can be proved in the fully continuous case. For this, one would need to assume that
the correlations of the disorder field $\{\xi(x,t): x\in \R^d, t \in \R\}$ are decaying sufficiently fast.

\

3. It is interesting to study the probability distribution for the limiting partition function $Z := Z_{x,s}^{\infty}$ and for $\Phi=-\ln{Z}$. Although these probability distributions are not
universal, we believe that the tail distributions have many universal features. We conjecture that in the case when the probability distribution
of $\xi$ has compact support, the left tail of the density for $\Phi$ behaves like $\exp{[-\Phi^{1+d/2}]}$ and the right tail decays like
$\exp{[-\Phi^{1+d}]}$. A related conjecture concerns the moments $m(l) :=\langle Z^l\rangle$ of $Z$, which we conjecture to grow as $\exp{[l^{1+2/d}]}$ in the limit
$l\to \infty$. If the disorder $\xi$ is Gaussian, then for any $\beta>0$ only a finite number of moments for $Z$ is finite. 
Thus, one can expect exponential decay of the left tail for $\Phi$. 

\

4. The uniqueness of global solutions to the stochastic heat equation and the random Hamilton-Jacobi equation was also
proved in dimension $d=1$ (\cite{BCK, BL1, BL2}). The mechanism leading to uniqueness in this case is completely different from our setting.
We should also mention that the case $d=1$ corresponds to the famous KPZ universality class.

\

The rest of this paper is organized as follows:
In Section~\ref{sec:setting} we derive an expansion for the partition functions and convergence to limiting partition functions. This allows us to state our main result, the factorization formula for $Z_{x,s}^{y,t}$.  We also derive asymptotics for both spatial and temporal correlations of the field of limiting partition functions.  
In Section~\ref{sec:transition_prob}, we collect several estimates on transition probabilities for the simple symmetric random walk on $\Z^d$.
Sections~\ref{sec:proof_lm_factorization} and~\ref{ssec:main_contribution_lemmas} are devoted to the proof of the factorization formula. 
Finally, in the appendix we prove the estimates on transition probabilities from Section~\ref{sec:transition_prob}.  

\textbf{Notation:} Throughout this article the Euclidean norm and inner product in $\R^d$ are denoted by $\| \cdot \|$ and $\ \cdot \ $, respectively. The $1$-norm in $\R^d$ is denoted by $\| \cdot \|_1$. We simply write $a \equiv b$ to indicate that $a \equiv b$ \ (mod 2). For functions $A$ and $B$, potentially of several variables, we write $A \lesssim B$ or $A$ is \emph{dominated} by $B$ to denote that $A \leq c B$ for some constant $c > 0$. The constant $c$ may depend on the dimension $d$, the inverse temperature $\beta$ and the law of the disorder (e.g., through $\lambda$ defined in~\eqref{eq:def_lambda}), or on scaling parameters such as $\sigma$ from Theorem~\ref{thm:factorization} or $\xi$ from Section~\ref{ssec:large_huge_gaps}. However, $c$ is not allowed to depend on any time or space variables such as $t$ and $z$. The same remark applies to every constant introduced in this paper. Finally, in order to simplify notation, we will write $\sum_{\zbf}$ to indicate that we are summing over all $\zbf = (z_1, \ldots, z_r) \in (\Z^d)^r$, where the value of $r$ will be clear from the context.

%======================================
\section*{Acknowledgments}
%======================================

Part of this paper was written during two-week stays at Mathematisches 
For\-schungs\-zentrum Oberwolfach, in 2018, and Centre International de Rencontres Math\'ematiques (CIRM), in 2019, as part of their respective research in pairs programs. We thank both institutions for their kind hospitality. TH gratefully acknowledges support through SNF grant $200021-175728/1$.

%======================================

%======================================
%======================================
%\subsection{}
%======================================
%======================================

%======================================
%======================================
\section{Setting and Main Result}
\label{sec:setting}
%======================================
%======================================

Let $\gamma = (\gamma_n)_{n \in \Z}$ be a discrete-time simple symmetric random walk on $\Z^d$, $d \geq 3$, starting at point $x \in \Z^d$ at time $s \in \Z$, with corresponding probability measure $\Pp_{x,s}$ and corresponding expectation $\E_{x,s}$.  
As $d \geq 3$, $\gamma$ is transient. 
For integers $t > s$ and $y \in \Z^d$, we denote the probability measure obtained from $\Psf_{x,s}$ by conditioning on the event $\{\gamma_t = y\}$ by $\Psf_{x,s}^{y,t}$.  
The corresponding expectation is denoted by $\E_{x,s}^{y,t}$.
We also set  
$$ 
	q_t^z \coleq \Pp_{0,0}(\gamma_t = z). 
$$
Let $(\xi(x, t))_{x \in \Z^d, t \in \Z}$ be a collection of i.i.d. random variables with corresponding probability measure $Q$ and corresponding expectation $\langle \:\cdot\: \rangle$. These constitute the random potential in our setting. We assume that 
$$
c(\beta) := \langle e^{\beta \xi(0,0)} \rangle < \infty
$$
for $\beta > 0$ sufficiently small. 
To a sample path of $\gamma$ over a time interval $[s,t]$, we assign the random action 
$$ 
\AC_s^t = \AC_s^t (\gamma) \coleq \sum_{j=s}^t  \xi(\gamma_j, j).  
$$ 
For integers $s < t$, $x, y \in \Z^d$, and inverse temperature $\beta > 0$, we define the random partition functions 
$$
Z_{x,s}^{y,t} \coleq c(\beta)^{-(t-s+1)} \ q_{t-s}^{y-x} \ \E_{x,s}^{y,t} e^{\beta \AC_s^t}, 
$$
$$ 
Z_{x,s}^t \coleq \sum_{y \in \Z^d} Z_{x,s}^{y,t}, \quad \text{and} \quad Z_s^{y,t} \coleq \sum_{x \in \Z^d} Z_{x,s}^{y,t}. 
$$
%
%We are interested in the high-temperature regime, i.e. $\beta$ is assumed to be small. 
Since $c(\beta)^{-(t-s+1)} \langle e^{\beta \AC_s^t} \rangle = 1$ for every realization of $\gamma$, we have $\langle Z_{x,s}^t \rangle = \langle Z_s^{y,t} \rangle = 1$. 
Notice that the law of the stochastic process $(Z_{x,s}^{s+\tau})_{\tau \in \N_0}$ with respect to $Q$ does not depend on $x$ or $s$. Besides, $(Z_{x,s}^{s+\tau})_{\tau \in \N_0}$ and $(Z_{t-\tau}^{y,t})_{\tau \in \N_0}$ have the same law. \\ 

%=========================================

\begin{remark}\rm 
This is essentially the model considered by Sinai, where $F(x,t)$ in~\cite{Sinai_95} corresponds to $\beta \xi(x,t)$ in our setting. Furthermore, the partition function $Z_{x,k}^{y,n}$ from~\cite{Sinai_95} becomes $c(\beta)^{n-k+1} Z_{x,k}^{y,n}$ in our notation. 
\end{remark} 

Given $z \in \Z^d$ and $s \in \Z$, define
\begin{equation*}  
	h(z, s) 
		\coleq \dfrac{e^{\beta \xi(z,s)} - c(\beta)}{c(\beta)}.
\end{equation*}
As shown in the proof of Theorem~2 in~\cite{Sinai_95}, $Z_{x,s}^{y,t}$ admits the expansion 
\begin{equation}     \label{eq:Z_bridge_expansion} 
	Z_{x,s}^{y,t}
		= q_{t-s}^{y-x} + \sum_{r=1}^{t-s+1} \sum_{\substack{s \leq i_1 < \ldots < i_r \leq t, \\ z_1, \ldots, z_r \in \Z^d}} q_{i_1-s}^{z_1 - x} q_{i_2 - i_1}^{z_2 - z_1} \ldots q_{i_r - i_{r-1}}^{z_r - z_{r-1}} q_{t- i_r}^{y- z_r} \prod_{j=1}^r h(z_j, i_j).
\end{equation}
Similarly, one obtains
\begin{equation}     \label{eq:Z_backwards} 
Z_s^{y,t} = 1 + \sum_{r=1}^{t-s+1} \sum_{\substack{s \leq i_1 < \ldots < i_r \leq t, \\ z_1, \ldots, z_r \in \Z^d}} q_{i_2 -i_1}^{z_2 - z_1} \ldots q_{t - i_r}^{y - z_r} \prod_{j=1}^r h(z_j, i_j) 
\end{equation} 
and 
\begin{equation}    \label{eq:Z_M} 
Z_{x,s}^{t} =  1 + \sum_{r=1}^{t-s+1} \sum_{\substack{s \leq i_1 < \ldots < i_r \leq t, \\ z_1, \ldots, z_r \in \Z^d}} q_{i_1-s}^{z_1 - x} \ldots q_{i_r - i_{r-1}}^{z_r - z_{r-1}} \prod_{j=1}^r h(z_j, i_j).
\end{equation}

\subsection{Convergence to limiting partition functions.}
\label{ssec:L2_convergence}

As in~\cite{Sinai_95}, define 
\begin{equation}    \label{eq:def_alpha} 
\alpha_d \coleq \sum_{t=1}^{\infty} \sum_{z \in \Z^d} \left(q_t^z \right)^2. 
\end{equation} 
It is well known that $q_t^z \lesssim t^{-\frac{d}{2}}$ for all $z \in \Z^d$ and $t \in \N$ (see for instance~\cite{Lawler_Limic}).
Therefore, as $d \geq 3$, there is a constant $C > 0$ such that
\begin{equation*} 
\alpha_d 
\leq C \sum_{t=1}^{\infty} \frac{1}{t^{\frac{d}{2}}} \sum_{z \in \Z^d} q_t^z 
= C\sum_{t=1}^{\infty} \frac{1}{t^{\frac{d}{2}}} 
< \infty. 
\end{equation*}
We also define 
\begin{equation}\label{eq:def_lambda} 
	\lambda \coleq c(\beta)^{-2} c(2 \beta) - 1. 
\end{equation} 

\medskip 

The following convergence statement for partition functions corresponds to Theorem~1 in~\cite{Sinai_95}. 

%==================================
\begin{theorem}     \label{thm:limiting_part_fun_exists}
For $\beta$ so small that $\alpha_d \lambda < 1$, the following holds: As $t \to \infty$, $Z_{x,s}^t$ converges in $L^2(Q)$ to a limiting partition function $Z_{x,s}^{\infty}$. 
\end{theorem} 

%==================================

\begin{remark}\rm Due to symmetry, we also have that
$$
Z_{-\infty}^{y,t}  := \lim_{s \to -\infty}  Z_s^{y,t}  
$$
exists in the sense of $L^2(Q)$ for all $y \in \Z^d$ and $t \in \Z$.  
\end{remark}

\begin{remark}     \rm 
As pointed out by Bolthausen (\cite{Bolthausen}), $(Z_{x,s}^t)_{t \geq s}$ is a martingale with respect to the filtration $\F_t := \sigma(\xi(y,u): s \leq u \leq t, y \in \Z^d)$, so convergence to the limiting partition functions also holds $Q$-almost surely by the martingale convergence theorem. 
\end{remark}

\bpf[Proof of Theorem~\ref{thm:limiting_part_fun_exists}] We follow the approach in~\cite{Sinai_95}. The right-hand side of~\eqref{eq:Z_M} has an orthogonality structure, which we shall exploit.
Since $h(z,s)$ and $h(z',s')$ are independent if $z \neq z'$ or if $s \neq s'$, and since $\langle h(z,s) \rangle = 0$, we have with Jensen's inequality and Fubini's theorem that $\langle (Z^t_{x,s})^2 \rangle$ is bounded from above by 
\begin{equation}      \label{eq:Sinai_3}
2 + 2 \sum_{r=1}^{\infty} \sum_{\substack{s \leq i_1 < \ldots < i_r, \\ z_1, \ldots, z_r \in \Z^d}} \left(q_{i_1-s}^{z_1-x} \right)^2 \ldots \left(q_{i_r - i_{r-1}}^{z_r - z_{r-1}} \right)^2 \biggl \langle \prod_{j=1}^r h(z_j, i_j)^2 \biggr \rangle. 
\end{equation}
Since $\langle h(z,s)^2 \rangle = c(\beta)^{-2} c(2 \beta) - 1$, we find
\begin{equation}\label{eq:eq_lambda}
\biggl \langle \prod_{j=1}^r h(z_j, i_j)^2 \biggr \rangle 
	= \prod_{j=1}^r \left(c(\beta)^{-2} c(2 \beta) - 1 \right)
	= \lambda^r. 
\end{equation}
Since $\alpha_d \lambda < 1$, one has $\sum_{r=1}^{\infty} (\alpha_d \lambda)^r < \infty$, so the expression in~\eqref{eq:Sinai_3} is finite. As a result, 
$$
\sup_{t > s} \left \langle \left(Z_{x,s}^t \right)^2 \right \rangle < \infty, 
$$
and $L^2$-convergence follows with the martingale convergence theorem. 
\epf

%=======================================================
%=======================================================
%=======================================================
\bigskip

The following theorem gives us a rate of convergence to the limiting partition function $Z_{x,s}^{\infty}$, which is needed to prove the factorization formula in Theorem~\ref{thm:factorization}.

%==================================
\begin{theorem}     \label{thm:limiting_part_fun}
For $\beta$ so small that $\alpha_d \lambda < 1$ and for $\theta \in (0, \min\{\tfrac{d}{2}-1,-\ln(\alpha_d \lambda)\})$, one has  
$$
\lim_{t \to \infty} (t-s)^{\theta} \left \langle \left(Z_{x,s}^t - Z_{x,s}^{\infty} \right)^2 \right \rangle = 0.
$$
\end{theorem} 

\begin{proof} For an integer $t \geq s$, let 
	\begin{equation*}
		M_t \coleq \left \langle \left(Z_{x,s}^t - 1 \right)^2 \right \rangle = \sum_{r=1}^{t-s+1} \lambda^r 
		\sum_{\substack{s \leq i_1 < \ldots < i_r \leq t, \\ z_1, \ldots, z_r \in \Z^d}} \left(q_{i_1 - s}^{z_1 - x} \right)^2 \ldots \left(q_{i_r - i_{r-1}}^{z_r - z_{r-1}} \right)^2,
	\end{equation*}
which is monotone increasing in $t$.  
Set 
	\begin{equation*}
		M :=\lim_{t \to \infty} M_t \in (0,+\infty).
	\end{equation*} 
Then, for $t > s$,  
\begin{align}    \label{eq:M_minus_M_t} 
\left \langle \left(Z_{x,s}^t - Z_{x,s}^{\infty} \right)^2 \right \rangle =& \lim_{T \to \infty} \left\langle \left(Z^t_{x,s} - Z^T_{x,s} \right)^2 \right\rangle 
\leq 2 (M - M_t) \\
\leq& 2 \sum_{1 \leq r \leq \ln(t-s)} \lambda^r \sum_{\substack{s \leq i_1 < \ldots < i_r, i_r > t, \\ z_1, \ldots, z_r \in \Z^d}} \left(q_{i_1 - s}^{z_1 - x} \right)^2 \ldots \left(q_{i_r - i_{r-1}}^{z_r - z_{r-1}} \right)^2  \notag \\
& + 2 \sum_{r > \ln(t-s)} \lambda^r \sum_{\substack{s \leq i_1 < \ldots < i_r, \\ z_1, \ldots, z_r \in \Z^d}} \left(q_{i_1 - s}^{z_1 - x} \right)^2 \ldots \left(q_{i_r - i_{r-1}}^{z_r - z_{r-1}} \right)^2. \notag    
	\end{align} 
The expression in the third line of~\eqref{eq:M_minus_M_t} is dominated by 
$$
\sum_{r > \ln(t-s)} (\alpha_d \lambda)^r \lesssim  (t-s)^{\ln(\alpha_d \lambda)}, 
$$
and 
$$
\lim_{t \to \infty} (t-s)^{\theta} (t-s)^{\ln(\alpha_d \lambda)} = 0, \quad \theta \in (0, -\ln(\alpha_d \lambda)). 
$$
The expression in the second line of~\eqref{eq:M_minus_M_t} is dominated by  
\begin{align*}
& \sum_{1 \leq r \leq \ln(t-s)} \lambda^r  \sum_{\substack{t_1, \ldots, t_r \in \N, \\ t_1 + \ldots + t_r > t-s}} \sum_{x_1, \ldots, x_r \in \Z^d} \left(q_{t_1}^{x_1} \right)^2 \ldots \left(q_{t_r}^{x_r} \right)^2 \\
\leq& \sum_{1 \leq r \leq \ln(t-s)} \lambda^r \sum_{l=1}^r \sum_{\substack{t_1, \ldots, t_r \in \N, \\ t_l \geq \frac{t-s}{\ln(t-s)}}} \prod_{k=1}^r \biggl(\sum_{x_k \in \Z^d} \left(q_{t_k}^{x_k} \right)^2 \biggr) \\
\lesssim& \sum_{r = 1}^{\infty} r (\alpha_d \lambda)^r \sum_{j \geq \frac{t-s}{\ln(t-s)}} \frac{1}{j^{\frac{d}{2}}} 
\lesssim \sum_{r=1}^{\infty} r (\alpha_d \lambda)^r \left(\frac{t-s}{\ln(t-s)} \right)^{1 - \frac{d}{2}},  
\end{align*}
and 
$$
\lim_{t \to \infty} (t-s)^{\theta} \left(\frac{t-s}{\ln(t-s)} \right)^{1-\frac{d}{2}} = 0, \quad \theta \in (0, \tfrac{d}{2}-1). 
$$

\end{proof}

%==================================

\subsection{Factorization formula.} 

The following factorization formula for the partition function $Z_{x,s}^{y,t}$ with fixed starting and endpoint is the main result of this article.  

%===============
\begin{theorem}     \label{thm:factorization}
Let $\beta$ be so small that $\alpha_d \lambda < 1$. For any $\sublin \in (0,1)$, no matter how close to $1$, there exists $\theta = \theta(\sigma) > 0$ such that for all $x,y\in\Z^d$ and $s < t$ with $\| x - y \| < (t - s)^\sublin$, the partition function $Z_{x,s}^{y,t}$ has the representation 
\begin{equation}     \label{eq:factor_formula}
Z_{x,s}^{y,t} = q_{t-s}^{y-x} \left(Z_{x,s}^{\infty} Z_{-\infty}^{y,t} + \delta_{x,s}^{y,t} \right),
\end{equation} 
where the error term $\delta_{x,s}^{y,t}$ defined by the formula above satisfies 
\begin{equation}   \label{eq:convergence_error}
\lim_{(t-s) \to \infty} (t-s)^{\theta} \sup_{x, y \in \Z^d: \|x-y\| < (t-s)^{\sublin}} \langle \lvert \delta_{x,s}^{y,t} \rvert \rangle = 0.
\end{equation}
\end{theorem}
%===============

Theorem~\ref{thm:factorization} is proved in Section~\ref{sec:proof_lm_factorization}. 
Notice that the formula is similar to the ones obtained by Sinai in~\cite[Theorem 2]{Sinai_95} and Kifer in~\cite[Theorem 6.1]{Kifer}.  
However, we show that the error term is small not only within the diffusive regime $\|x-y\| < O(t-s)^{\frac{1}{2}}$, but also for $\|x-y\| < (t-s)^{\sublin}$ with $\sublin$ arbitrarily close to $1$. 
This extension beyond the diffusive regime is nontrivial because the error term in~\eqref{eq:factor_formula} is multiplied by the random-walk transition probability $q_{t-s}^{y-x}$, which is itself extremely small for $\|x-y \| \geq (t-s)^{\frac{1}{2}}$. 
In a forthcoming publication, we rely heavily on a continuous-time version of Theorem~\ref{thm:factorization} to prove a uniqueness statement for global solutions to the semi-discrete stochastic heat equation.

\subsection{Correlations for the field of limiting partition functions.} 

As mentioned in Section~\ref{sec:intro}, the distribution for the field of limiting partition functions $(Z_{x,s}^{\infty})_{x \in \Z^d, s \in \Z}$ is an interesting object to study, with several important questions still open. Below, we state asymptotics for the spatial and temporal correlations of this field.

\begin{theorem}     \label{thm:correlations_Z} 
Let $\beta$ be so small that $\alpha_d \lambda < 1$. Then the spatial and temporal correlations for the field of limiting partition functions $(Z_{x,s}^{\infty})_{x \in \Z^d, s \in \Z}$ have the following asymptotics. 
\begin{enumerate}
\item $$
\lim_{\substack{\|y\| \to \infty, \\ \|y\|_1 \equiv 0}} \|y\|^{d-2} \left(\langle Z_{0,0}^{\infty} Z_{y,0}^{\infty} \rangle - \langle Z_{0,0}^{\infty} \rangle \langle Z_{y,0}^{\infty} \rangle \right) \in (0,\infty); 
$$
\item $$
\lim_{\substack{\lvert s \rvert \to \infty, \\ s \equiv 0}} \lvert s \rvert^{\frac{d}{2} - 1} \left(\langle Z_{0,0}^{\infty} Z_{0,s}^{\infty} \rangle - \langle Z_{0,0}^{\infty} \rangle \langle Z_{0,s}^{\infty} \rangle \right) \in (0, \infty). 
$$
\end{enumerate}
\end{theorem} 

It is necessary to take the limit in part $(1)$ along sequences $(y_n)$ such that $\|y_n\|_1 \equiv 0$ for all $n$, for $Z_{0,0}^{\infty}$ and $Z_{y,0}^{\infty}$ are independent if $\| y\|_1 \equiv 1$. A similar observation applies to the limit in part $(2)$. The proof of Theorem~\ref{thm:correlations_Z} relies on the following estimates for simple symmetric random walk on $\Z^d$, $d \geq 3$. 

\begin{lemma}   \label{lm:correlations_Z} 
The following statements hold: 
\begin{enumerate} 
\item $$
\lim_{\substack{\| y\| \to \infty, \\ \|y\|_1 \equiv 0}} \|y\|^{d-2} \sum_{t=0}^{\infty} \sum_{x \in \Z^d} q_t^x q_t^{y-x} \in (0,\infty); 
$$
\item $$
\lim_{\substack{s \to \infty, \\ s \equiv 0}} s^{\frac{d}{2}-1} \sum_{t=0}^{\infty} \sum_{x \in \Z^d} q_t^x q_{s+t}^x \in (0, \infty). 
$$
\end{enumerate}
\end{lemma} 

\bpf For $y \in \Z^d$ whose $1$-norm is even, 
$$
\sum_{t=0}^{\infty} \sum_{x \in \Z^d} q_t^x q_t^{y-x} = \sum_{t=0}^{\infty} q_{2t}^y = G(0,y), 
$$
where $G$ denotes the Green's function for simple symmetric random walk on $\Z^d$. Theorem 4.3.1 in~\cite{Lawler_Limic} implies that 
$$
\lim_{\substack{\|y\| \to \infty, \\ \|y\|_1 \equiv 0}} \|y\|^{d-2} G(0,y) \in (0,\infty), 
$$
so $(1)$ follows. To prove $(2)$, first notice that for every even $s \in \N_0$, 
$$
\sum_{t=0}^{\infty} \sum_{x \in \Z^d} q_t^x q_{s+t}^x = \sum_{t=0}^{\infty} q_{s+2t}^0 = \sum_{t=s/2}^{\infty} q_{2t}^0. 
$$ 
It is well known (see, e.g., Chapter 1 of~\cite{Lawler_HE}) that 
$$
\lim_{t \to \infty} t^{\frac{d}{2}} q_{2t}^0 =: c \in (0, \infty). 
$$
Let $\eps > 0$. Then there is $T \in \N$ such that 
$$
c - \eps \leq t^{\frac{d}{2}} q_{2t}^0 \leq c + \eps, \quad \forall t \geq T. 
$$
Thus, for even $s \geq 2T$, 
$$
\sum_{t = s/2}^{\infty} q_{2t}^0 \leq (c+\eps) \sum_{t = s/2}^{\infty} t^{-\frac{d}{2}} \leq (c+\eps) \frac{2}{d-2} \left(\frac{s}{2} - 1 \right)^{1-\frac{d}{2}}
$$
and 
$$
\sum_{t = s/2}^{\infty} q_{2t}^0 \geq (c-\eps) \frac{2}{d-2} \left(\frac{s}{2} \right)^{1-\frac{d}{2}}. 
$$
Hence, 
$$
\limsup_{\substack{s \to \infty, \\ s \equiv 0}} s^{\frac{d}{2}-1} \sum_{t = s/2}^{\infty} q_{2t}^0 \leq (c+ \eps) \frac{2^{\frac{d}{2}}}{d-2} \quad \text{and} \quad 
\liminf_{\substack{s \to \infty, \\ s \equiv 0}} s^{\frac{d}{2}-1} \sum_{t = s/2}^{\infty} q_{2t}^0 \geq (c-\eps) \frac{2^{\frac{d}{2}}}{d-2}. 
$$
Since $\eps$ was arbitrarily chosen, we obtain $(2)$. 
\epf 

\bigskip

\bpf[Proof of Theorem~\ref{thm:correlations_Z}] For $y \in \Z^d$ such that $\|y\|_1 \equiv 0$ and $t \in \N$, the expansion in~\eqref{eq:Z_M} along with the properties of $h(z,s)$ yield   
$$
\langle Z_{0,0}^t Z_{y,0}^t \rangle = 1 + \sum_{r=1}^{t+1}  \lambda^r  \sum_{\substack{0 \leq i_1 < \ldots < i_r \leq t, \\ z_1, \ldots, z_r \in \Z^d}}  q_{i_1}^{z_1} q_{i_1}^{z_1 - y} \left(q_{i_2 - i_1}^{z_2 - z_1} \right)^2  \ldots \left( q_{i_r - i_{r-1}}^{z_r - z_{r-1}} \right)^2.  
$$
Along the lines of the proof of Theorem~\ref{thm:limiting_part_fun}, one can easily show that, as $t \to \infty$, the expression on the right converges to 
$$
1 + \sum_{r=1}^{\infty} \lambda^r \sum_{\substack{0 \leq i_1 < \ldots < i_r, \\ z_1, \ldots, z_r \in \Z^d}} q_{i_1}^{z_1} q_{i_1}^{z_1 - y} \left(q_{i_2-i_1}^{z_2-z_1} \right)^2 \ldots \left(q_{i_r - i_{r-1}}^{z_r - z_{r-1}} \right)^2. 
$$
Therefore, 
\begin{align*}
&\langle Z_{0,0}^{\infty} Z_{y,0}^{\infty} \rangle - \langle Z_{0,0}^{\infty} \rangle \langle Z_{y,0}^{\infty} \rangle \\
=& \sum_{r=1}^{\infty} \lambda^r \sum_{\substack{0 \leq i_1 < \ldots < i_r, \\ z_1, \ldots, z_r \in \Z^d}} q_{i_1}^{z_1} q_{i_1}^{z_1 - y} \left(q_{i_2 - i_1}^{z_2 - z_1} \right)^2 \ldots \left(q_{i_r -i_{r-1}}^{z_r - z_{r-1}} \right)^2 \\
=& \sum_{i=0}^{\infty} \sum_{x \in \Z^d} q_i^x q_i^{x-y} \alpha_d^{-1} \sum_{r=1}^{\infty} (\alpha_d \lambda)^r, 
\end{align*}
and part $(1)$ follows from Lemma~\ref{lm:correlations_Z}.  The proof of part $(2)$ is similar and we omit it. 
\epf 

\bigskip

%======================================
%======================================
\section{Transition Probabilities for the Simple Symmetric Random Walk}
\label{sec:transition_prob}
%======================================
%======================================
In this section we collect several estimates on transition probabilities for the discrete-time simple symmetric random walk on $\Z^d$. Subsection~\ref{ssec:transition_prob_proofs} will be devoted to the proofs of the results presented here.

Let $(\gamma_n)_{n \in \N_0}$ be a discrete-time simple symmetric random walk on $\Z^d$ starting at the origin.

%================
\begin{lemma}    \label{lm:Law_Lim}
There are constants $c_1, c_2 > 0$ such that the following holds: For every $\sublin \in (\tfrac{3}{4},1)$ and $\tilde{\sublin} \in (\sublin, 1)$, there exists $T \in \N$ such that for every $t \geq T$ and $y \in \Z^d$ with $q^y_t > 0$ and $\|y\| \leq t^{\sublin}$,  
\begin{equation}  \label{eq:lclt_lower_bound}
	q^y_t 
		\geq 
		c_1 \left(\tfrac{d}{2 \pi t} \right)^{\frac{d}{2}} 
		\exp \big( -\tfrac{d}{2t} \|y\|^2 \big)
		\exp \Big(-c_2 t^{4 \tilde \sublin - 3} \Big).  
\end{equation}
\end{lemma}
%================

%======================================
%======================================
%\subsection{Nazarov's Estimates}
%======================================
%======================================

%\RED{Say a word about Nazarov's contribution}
%======================================

%================
\begin{lemma}     \label{lm:lin_func_upper_bd}
There is $c_1 > 0$ such that for every $y \in \Z^d$ and for every linear functional $\varphi$ on $\R^d$ with $\lvert \varphi(x) \rvert \leq \|x\|$, $x \in \R^d$, we have 
$$ 
q_t^y e^{\varphi(y)} \leq c_1 t^{-\frac{d}{2}} \sum_{z \in \Z^d} q_t^z e^{\varphi(z)}, \quad \forall t \in \N. 
$$ 
In particular, 
	\begin{equation*}  
		q_t^y \lesssim t^{-\frac{d}{2}}, \quad \forall t \in \N, y \in \Z^d. 
	\end{equation*}
\end{lemma}
%======================================

Fix a linear functional $\varphi$ on $\R^d$ such that $\lvert \varphi(x) \rvert \leq \|x\|$ for all $x \in \R^d$.
To simplify notation, we set $\varphi_j \coleq \varphi(e_j)$ for $1 \leq j \leq d$, where $\{ e_j \}$ is the standard basis in $\R^d$.
Define, for all $\theta = (\theta^1, \ldots, \theta^d) \in \R^d$,
\begin{equation}    \label{eq:phi_def}
	\Phi (\theta) 
		\coleq \E \left[ e^{i \theta \cdot \gamma_1} e^{\varphi(\gamma_1)} \right] 
		= \frac{1}{2d} \sum_{j=1}^d \left(e^{i \theta^j} e^{\varphi_j} + e^{-i \theta^j} e^{-\varphi_j} \right), 
\end{equation}
where $i$ is the imaginary unit.  
Notice that for all $\theta \in \R^d$,
\begin{equation}\label{eq:phi_theta_bd}
	\big| \Phi (\theta) \big|
		\leq \Phi (0) = \sum_{z \in \Z^d} q_1^z e^{\varphi(z)}, 
\end{equation}
where $0$ is the zero vector in $\R^d$. 
Furthermore,
\begin{equation}\label{190918172610}
	\Phi (0)^t = \sum_{y \in \Z^d} q_t^y e^{\varphi (y)}, \quad \forall t \in \N_0.
\end{equation}
Notice also that $\Phi$ is $2\pi$-periodic in every argument, so it will be convenient to work with the cube $\Cc \coleq (-\tfrac{\pi}{2}, \tfrac{3 \pi}{2}]^d$.
It is not hard to see that the inequality~\eqref{eq:phi_theta_bd} is strict for all $\theta \in \Cc$ except for $\theta^0 \coleq (0, \ldots, 0)$ and $\theta^1 \coleq (\pi, \ldots, \pi)$.

%For $z \in \Z^d$ and $n \in \N$, let $\widehat{\Phi^n}$ be the Fourier transform of $\Phi^n$; i.e.,
%$$ 
%\widehat{\Phi^n}(z) \coleq \frac{1}{(2 \pi)^d} \int_{\Cc} \Phi(\theta)^n e^{-i \langle \theta, z \rangle} \, d \theta.
%$$
%Since $\Phi (\theta)^n = \E \left[ e^{i \langle \theta, \gamma_n \rangle} e^{\varphi(\gamma_n)} \right],$ we get
%\begin{equation}\label{eq:Fourier} 
%	\widehat{\Phi^n}(z) 
%		= \sum_{y \in \Z^d} \Pp(\gamma_n = y)
%			e^{\varphi(y)} \frac{1}{(2 \pi)^d} 
%			\int_{\Cc} e^{i \langle \theta, y-z \rangle} \, d \theta
%		= q_n^z e^{\varphi(z)}.
%\end{equation}

%======================================
\begin{lemma}     \label{lm:linear_functional}
There are $\rho_1, \rho_2 > 0$ such that the following holds: For any $t \in \N$ and for any $z \in \Z^d$ such that $\|z\| \leq \rho_1 t$ and $q^z_t > 0$, there is a linear functional $\varphi$ on $\R^d$ of norm $\|\varphi\| \leq \rho_2 \tfrac{\|z\|}{t}$ which satisfies
$$
\frac{1}{(2 \pi)^d} \int_{\Cc} 
	\big| \Phi (\theta) \big|^t \, d \theta 
	\leq \left(1 + O(t^{-\frac{2}{5}}) \right) q^z_t e^{\varphi(z)}
$$
and
$$
	q_t^z e^{\varphi(z)} 
		\gtrsim t^{-\frac{d}{2}} \sum_{y \in \Z^d} q_t^y e^{\varphi(y)}. 
$$ 
\end{lemma}
%======================================

%======================================
\begin{lemma}    \label{lm:ratio_estimate}
There are constants $\rho, c > 0$ such that for any $t, t' \in \N$ and for any $z, z' \in \Z^d$ with $\|z\| \leq \rho t$ and $q^z_t > 0$, we have 
$$ 
\frac{q_{t'}^{z'}}{q_t^z} \leq  \left(1+O(t^{-\frac{2}{5}}) \right) \exp \left(c \left(\frac{\|z\|}{t} \left(\|z-z'\| + \lvert t' - t \rvert \right) + \ln(t) \frac{\lvert t-t' \rvert}{t} \right) \right). 
$$
\end{lemma}
%======================================

%======================================
%======================================

%======================================
%======================================

%======================

%======================================
%======================================

%======================================
%======================================
\section{Proof of Theorem \protect\ref{thm:factorization}}
\label{sec:proof_lm_factorization}
%======================================
%======================================
The main idea behind the factorization formula, which goes back at least to~\cite{Sinai_95}, is that there is strong averaging for times neither too close to $s$ nor too close to $t$. 

For fixed $i_1, \ldots, i_r$ and $z_1, \ldots, z_r$, the random walk is pinned to the points $z_1, \ldots, z_r$ at the corresponding times $i_1, \ldots, i_r$. The proof of Theorem~\ref{thm:limiting_part_fun_exists} suggests that the contribution to $Z_{x,s}^{y,t}$ from $r$ on the order of $(t-s)$ is negligible. If $r$ is not on the order of $(t-s)$, at least one of the gaps $i_j - i_{j-1}$ must be in some sense large (see Subsection~\ref{ssec:large_huge_gaps}). In Subsection~\ref{190918135426}, we show that the contribution to $Z_{x,s}^{y,t}$ coming from two or more large gaps is negligible as well. Thus, the main contribution comes from having exactly one large gap $i_j - i_{j-1}$, which is then on the order of $(t-s)$. In order for $q_{i_j - i_{j-1}}^{z_j - z_{j-1}}$ to be positive, $z_{j-1}$ must be close to $x$ and $z_j$ must be close to $y$. The transition probability $q_{i_j - i_{j-1}}^{z_j - z_{j-1}}$ is then close to $q_{t-s}^{y-x}$.\\

Notice that to prove Theorem~\ref{thm:factorization}, it is enough to show that for $\alpha_d \lambda < 1$, for any given $\sigma \in (0,1)$ there is $\theta > 0$ such that 
$$ 
\lim_{t \to \infty} t^{\theta} \sup_{y \in \Z^d: \|y\| < t^{\sublin}} \langle \lvert \delta_{0,0}^{y,t} \rvert \rangle = 0. 
$$ 
This is because for a fixed realization $\omega$ of the disorder, $\delta_{x,s}^{y,t}(\omega)$ can be written as $\delta_{0,0}^{y-x,t-s}(\hat \omega)$, where $\hat \omega$ is obtained by shifting $\omega$ in space and time. The distribution of the disorder is invariant under such shifts.

For $t \in \N_0$ and $r \in \{1, \ldots, t+1\}$, let 
$$
I(t,r) \coleq \{\ibf = (i_1,\ldots,i_r) \in \N_0^r \,:\, 0 \leq i_1 < \cdots < i_r \leq t\}. 
$$
For $\ibf \in I(t,r)$ and $\zbf = (z_1, \ldots, z_r) \in (\Z^d)^r$, define 
%===
\begin{equation*}
q_t^y(\ibf,\zbf) \coleq q_{i_1}^{z_1} q_{i_2 - i_1}^{z_2 - z_1} \ldots q_{t - i_r}^{y-z_r}.
\end{equation*} 
With this notation, the expansion in~\eqref{eq:Z_bridge_expansion} becomes  
\begin{equation}     \label{eq:shorthand_Z_expansion} 
	Z_{0,0}^{y,t} 
		=  q_{t}^{y} +  \sum_{r=1}^{t+1} \sum_{\ibf \in I(t,r), \zbf} q_{t}^y(\ibf,\zbf) \prod_{j=1}^r h(z_j, i_j), 
\end{equation} 
where one should recall from Section~\ref{sec:intro} the notational shorthand $\sum_{\zbf}$ for summation over all $\zbf = (z_1, \ldots, z_r) \in (\Z^d)^r$.  The first step is to split the double sum into terms according to the size of the largest gap between indices, as discussed in Subsection~\ref{ssec:large_huge_gaps}.

%======================================
%======================================
\subsection{Large and huge gaps}
\label{ssec:large_huge_gaps}
%======================================
%======================================
If there are $\sigma \in (0,1)$ and $\theta > 0$ such that~\eqref{eq:convergence_error} (the convergence of the error term in the factorization formula) holds, then~\eqref{eq:convergence_error} also holds for all $\tilde \sigma \in (0, \sigma)$ and the same $\theta$. There is then no loss of generality in assuming that $\sigma > 3/4$, and one may even think of $\sigma$ as being very close to $1$. For a collection of indices $0 \eqcol i_0 \leq i_1 < \ldots < i_r \leq i_{r+1} \coleq t$, the \emph{gaps} are the differences between consecutive indices, i.e. $i_1 - i_0, i_2 - i_1, \ldots, i_{r+1} - i_r$. To quantify what it means to have many gaps, we fix positive constants $\kappa_1$, $\kappa_2 \in (\tfrac{1}{2} (3 \sigma - 1), \sigma)$ such that $\kappa_1 < \kappa_2$. Let $T_{\kappa_2} \in \N$ be so large that 
$
2 (t - t^{\kappa_2}) > t 
$
for all $t \geq T_{\kappa_2}$.
Then we define
$$
k(t) \coleq \begin{cases}
             (t-T_{\kappa_2} )^{\kappa_1} - 1, & \quad (t-T_{\kappa_2} )^{\kappa_1} - 1 \geq 1, \\
                                       0, & \quad (t-T_{\kappa_2})^{\kappa_1} - 1 < 1. 
                                       \end{cases}
$$
Note that $k(t)$ grows with $t$ like $t^{\kappa_1}$.
We say that a collection of indices $0 \leq i_1 < \ldots < i_r \leq t$ has \emph{many gaps} if $r > k(t)$.

To classify the size of a gap between indices, fix another constant $\xi$ such that $0 < \xi <  \min \big\{ 1-\sublin, \kappa_2 - \kappa_1 \}$.
One should think of $\xi$ as being very close to $0$. Note that $\xi + \kappa_1 < 1$ and that $\xi < \kappa_1$, the latter because of $\xi < 1 - \sigma < 1/4 < \tfrac{1}{2} (\tfrac{9}{4} - 1) < \kappa_1$.  
Let $t \in \N$ such that $k(t) \geq 1$, $r$ such that $1 \le r \le k(t)$, and consider a sequence of indices $0 = i_0 \leq i_1 < \ldots < i_r \le i_{r+1} = t$.
We say that the gap between two consecutive indices $i_{j-1}$ and $i_j$ is
\begin{itemize}
\item \emph{large} if $i_j - i_{j-1} \geq t^{\xi}$;
\item \emph{huge} if $i_j - i_{j-1} \geq t - r t^{\xi}$.
\end{itemize}
Observe that the size of the largest gap is necessarily greater than $t / (r+1) \geq t^{1-\kappa_1} \geq t^{\xi}$, so there is at least one large gap.
A huge gap is necessarily large.
If there is only one large gap, then all other gaps are of size less than $t^{\xi}$, so this large gap is even huge.
Thus, if there is no huge gap, there are at least two large ones.
Since $t$ must be greater than $T_{\kappa_2}$ in order for $k (t) \geq 1$ to hold, we have 
$
2(t-r t^{\xi}) > 2 (t-t^{\kappa_1+\xi}) > 2 (t-t^{\kappa_2}) > t,
$
so there can be at most one huge gap.
Note, however, that a huge gap is not necessarily the only large one.

%\RED{
%Explain the meaning/importance of huge gaps.
%}

Let us introduce some more notation.
Fix $r \in \N$ and $t \in \N_0$.
For any $m \in \N$ such that $1 \leq m \leq r+1$, define the following set of $r$-tuples:
$$
	I_1 (t, r, m)
		\coleq
		\left\{ (i_1, \ldots, i_r)  \in I(t,r) ~:~
			\text{the gap between $i_{m-1}$ and $i_m$ is huge}
			\right\}.
$$
Also define
$$
	I_2 (t, r)
		\coleq
		\left\{ (i_1, \ldots, i_r) \in I(t,r) ~:~
			\text{there is no huge gap}
			\right\}.
$$
For $t$ so large that $k(t) \geq 1$, we decompose the expansion of $Z^{y,t}_{0,0}$ in~\eqref{eq:shorthand_Z_expansion} as follows:
$$ 
Z_{0,0}^{y,t} = q_t^y + \sum_{j=1}^3 B_j^{y,t}, 
$$
where,
\begin{align*}
B_1^{y,t} \coleq &  \sum_{k(t) < r \leq t +1} \sum_{\ibf \in I(t,r), \zbf} q_{t}^y(\ibf,\zbf)  \prod_{j=1}^r h(z_j, i_j), \\
B_2^{y,t} \coleq & \sum_{1 \leq r \leq k(t)} \sum_{\ibf \in I_2 (t, r), \zbf} q_{t}^y(\ibf,\zbf)  \prod_{j=1}^r h(z_j, i_j), \\
B_3^{y,t}  \coleq & \sum_{1 \leq r \leq k(t)} \sum_{m=1}^{r+1} \sum_{\ibf \in I_1 (t, r, m), \zbf} q_{t}^y(\ibf,\zbf)  \prod_{j=1}^r h(z_j, i_j).  
\end{align*}
With this decomposition in hand, Theorem~\ref{thm:factorization} follows immediately from the following lemma.

\begin{lemma}[Central Lemma]
\label{lm:central_lemma}
Let $\beta >0$ be so small that $\alpha_d \lambda < 1$, and let $\sublin \in (0,1)$. 
\begin{enumerate}
\item For every $\theta > 0$,	
\begin{equation}
\label{190918115555}
	\lim_{t \to \infty} t^{\theta} 
	\sup_{y: \|y\| \leq t^{\sublin}, q^y_t > 0} 
	\frac{\langle \lvert B_1^{y,t} \rvert \rangle}{q_t^y} 
		= 0. 
\end{equation}
%===
\item There is $\theta > 0$ such that 
\begin{equation}
\label{190918141213}	
\lim_{t \to \infty} t^{\theta} \sup_{y: \|y\| \leq t^{\sublin}, q^y_t > 0} \frac{\langle \lvert B^{y,t}_2 \rvert \rangle}{q_t^y} = 0. 
\end{equation}
%===
\item There is $\theta > 0$ such that 
\begin{equation}    \label{eq:convergence_3}
\lim_{t \to \infty} t^{\theta} \sup_{y: \|y\| \leq t^{\sublin}, q_t^y > 0} \left \langle \left \lvert 1 + \frac{B_3^{y,t}}{q_t^y} - Z_{0,0}^{\infty} Z_{-\infty}^{y,t} \right \rvert \right \rangle = 0. 
\end{equation}
\end{enumerate}
\end{lemma}

The Sections~\ref{ssec:proof_central_lemma_part12} and~\ref{ssec:proof_central_lemma_part3} are devoted to the proof of this lemma.

%======================================
%======================================
\subsection{Proof of the Central Lemma, Parts 1 and 2: Small contributions}
\label{ssec:proof_central_lemma_part12}
%======================================
%======================================

In this subsection, we show that the contributions of the terms $B_1^{y,t}$ and $B_2^{y,t}$ to $Z_{0,0}^{y,t}$ are negligible.
We start with the observation that, by Jensen's inequality, 
\begin{equation}     \label{eq:Jensen_B} 
	\left( 
		\frac{1}{q^y_t} 
		\Big\langle \big| B_j^{y,t} \big| \Big\rangle
	\right)^2
		\leq \frac{1}{(q^y_t)^2} \left\langle 
				\left(B_j^{y,t} \right)^2 \right\rangle, \quad j = 1,2. 
\end{equation} 

%======================================
%======================================
\subsubsection{Proof of Part 1: Many gaps}
\label{190918135324}
%======================================
%======================================

Let $t \in \N$ be so large that $k(t) \geq 1$. Since $\alpha_d \lambda < 1$,~\eqref{eq:eq_lambda} and the definition~\eqref{eq:def_alpha} of $\alpha_d$ let us estimate $\langle (B_1^{y,t})^2 \rangle$ as follows:
\begin{align*}
 \left \langle \left(B_1^{y,t} \right)^2 \right \rangle =& \sum_{k(t) < r \leq t+1} \lambda^r \sum_{\ibf \in I(t,r), \zbf} q^y_t(\ibf, \zbf)^2 \\
  \lesssim& \sum_{k(t) < r \leq t + 1}
				(\alpha_d \lambda)^r
		\leq
			\sum_{r > k(t)} (\alpha_d \lambda)^r
		\leq
			\frac{(\alpha_d \lambda)^{k(t)}}{1 - \alpha_d \lambda}.
\end{align*}
Recall our assumption that $\sigma > \tfrac{3}{4}$. To estimate $1/(q_t^y)^2$ on the right-hand side of~\eqref{eq:Jensen_B}, fix $\tilde{\sublin} \in (\sigma, 1)$ such that $4 \tilde{\sublin} - 3 < 2 \sublin - 1$. 
By Lemma~\ref{lm:Law_Lim}, there are constants $c_1, c_2 > 0$ (independent of $\sublin, \tilde{\sublin}$) and $T \in \N$ (depending on $\sublin,\tilde{\sublin}$) such that for every integer $t \geq T$ and $y \in \Z^d$ with $q_t^y > 0$ and $\| y \| \leq t^{\sublin}$, \begin{flalign*}
	q^y_t 
		&\geq c_1 
		\left(\tfrac{d}{2 \pi t} \right)^{d/2} 
		\exp \big( -\tfrac{d}{2t} \|y\|^2 \big)
			\exp \left(-c_2 t^{4 \tilde{\sublin} - 3} \right)
	\\	&\gtrsim t^{-d/2}
		\exp \Big( - \tfrac{d}{2} t^{2 \sublin - 1} 
						- c_2 t^{4 \tilde{\sublin} - 3} \Big)
		\geq t^{-d/2}
		\exp \big( - c t^{2 \sublin - 1} \big)
\end{flalign*}
for some constant $c > 0$. Therefore, 
$$
\sup_{y: \| y\| \leq t^{\sigma}, q^y_t > 0} \frac{1}{(q^y_t)^2} 	\left \langle \left(B_1^{y,t} \right)^2 \right \rangle 
		\lesssim
		t^d \ (\alpha_d \lambda)^{k(t)} \ 
		\exp \big( 2 c t^{2 \sublin - 1} \big).
$$
Since $\kappa_1 > \tfrac{1}{2} (3 \sublin -1) > 2 \sublin -1$, we have $t^{2 \sublin -1} / k(t) \to 0$ as $t \to \infty$, and therefore, for all $\theta > 0$, 
$$
	t^{\theta} \sup_{y: \|y\| \leq t^{\sublin}, q^y_t > 0} \frac{1}{(q^y_t)^2} \left \langle \left(B^{y,t}_1 \right)^2 \right \rangle 
	\lesssim
	t^{\theta+d} \ (\alpha_d \lambda)^{k(t)}  \ 
		\exp \big( 2 c t^{2 \sublin - 1} \big) 
		 \xrightarrow[t \to \infty]{} 0.
$$

%======================================
%======================================
\subsubsection{Proof of Part 2: No huge gaps}
\label{190918135426}
%======================================
%======================================

Let $t \in \N$ be so large that $k(t) \geq 1$. Then 
\begin{equation}    \label{eq:two_gaps_1} 
\left \langle \left(B_2^{y,t} \right)^2 \right \rangle = \sum_{1 \leq r \leq k(t)} \lambda^r \sum_{\ibf \in I_2(t,r), \zbf} q^y_t(\ibf, \zbf)^2 \lesssim \sum_{1 \leq r \leq k(t)} \lambda^r M_{t,r}(y), 
\end{equation} 
where 
$$
M_{t,r}(y) \coleq \sum_{\substack{\ibf \in I_2(t,r), \zbf \\ i_1 \neq 0, i_r \neq t}} q^y_t(\ibf, \zbf)^2. 
$$
Now we estimate $M_{t,r}(y)$. Let $r \in \N$ such that $1 \leq r \leq k(t)$, and $y \in \Z^d$ such that $\| y \| \leq t^{\sublin}$ and $q^y_t > 0$. Given $\ibf = (i_1, \ldots, i_r) \in I_2(t,r)$ such that $i_1 \neq 0$ and $i_r \neq t$, set $t_1 \coleq i_1, t_2 \coleq i_2 - i_1, \ldots, t_r \coleq i_r - i_{r-1}, t_{r+1} \coleq t - i_r$. And given $\zbf = (z_1, \ldots, z_r) \in (\Z^d)^r$, set $x_1 \coleq z_1, x_2 \coleq z_2 - z_1, \ldots, x_r \coleq z_r - z_{r-1}, x_{r+1} \coleq y - z_r$. This change of variables yields 
\begin{equation}     \label{eq:cov_q} 
q_t^y(\ibf, \zbf)^2 = \left(q_{t_1}^{x_1} \right)^2 \ldots \left(q_{t_{r+1}}^{x_{r+1}} \right)^2. 
\end{equation} 
For $\ibf \in I_2(t,r)$, there is no huge gap and hence there are at least two large ones. Let 
$$
l \coleq \left \lvert \left\{1 \leq j \leq r+1: \ t_j \geq t^{\xi} \right\} \right \rvert - 1,  
$$
i.e. $(l+1)$ gives the number of large gaps in $\ibf$. There are $(r+1)$ possible slots for the largest gap (which is then also a \emph{large gap}), and $\binom{r}{l}$ possible slots for the other $l$ large gaps once the largest gap has been fixed. Together with~\eqref{eq:cov_q}, this gives the estimate 
\begin{equation}    \label{eq:M_upper_bd} 
M_{t,r}(y) \leq (r+1) \sum_{l=1}^r \binom{r}{l} M_{t,r,l}(y), 
\end{equation} 
where 
$$
M_{t,r,l}(y) \coleq \qquad \sum_{\mathclap{\substack{
                                       \\  t_1 + \ldots + t_{r+1} = t 
                                       \\ x_1 + \ldots + x_{r+1} = y 
                                        \\ t_{r+1} \geq t_1, \ldots, t_l \geq t^{\xi} 
                                        \\ t_{l+1}, \ldots, t_r < t^{\xi}
                                        }}} 
                                        \qquad \left(q_{t_1}^{x_1} \right)^2 \ldots \left(q_{t_{r+1}}^{x_{r+1}} \right)^2. 
$$
The sum on the right-hand side is taken over all $t_1, \ldots, t_{r+1} \in \N$ and $x_1, \ldots, x_{r+1} \in \Z^d$ that satisfy the four conditions under the summation sign. In the special case $l=r$, the fourth condition is void. 

For given positive integers $t_{l+1}, \ldots, t_r$ that are strictly less than $t^{\xi}$, set 
$$
t'(t_{l+1}, \ldots, t_r) = t' \coleq t - (t_{l+1} + \ldots + t_{r}); 
$$
and for $x_{l+1}, \ldots, x_r \in \Z^d$ set 
$$
x'(x_{l+1}, \ldots, x_r) = x'  \coleq y - (x_{l+1} + \ldots + x_{r}).
$$
If $l < r$, this lets us write 
\begin{flalign}
\label{190918154637}
	M_{t,r,l} (y)
		& = \quad
		\sum_{\mathclap{\substack{
%			n_{l+1}, \ldots, n_r \in \N
%		\\	z_{l+1}, \ldots, z_r \in \Z^d
		\\	t_{l+1}, \ldots, t_{r} < t^\xi
%		\\	n' \coleq n - (n_{l+1} + \ldots + n_{r})
%		\\	z' \coleq y - (z_{l+1} + \ldots + z_{r})
			}}} \;\;
			\big( q_{t_{l+1}}^{x_{l+1}} \big)^2 
				\cdots 
			\big( q_{t_r}^{x_r} \big)^2
			M_{t,r,l}^{t'} (x'),
\\	\label{190918155153}
\text{where} \quad	M_{t,r,l}^{t'} (x')
		&\coleq \qquad
		\sum_{\mathclap{ \substack{
%			n_1, \ldots, n_l, n_{r+1} \in \N
%		\\	z_1, \ldots, z_l, z_{r+1} \in \Z^d
		\\	t_1 + \ldots + t_l + t_{r+1} = t'
		\\	x_1 + \ldots + x_l + x_{r+1} = x'
		\\	t_{r+1} \geq t_1, \ldots, t_l \geq t^\xi
			}}} \qquad
			\big( q_{t_1}^{x_1} \big)^2
				\cdots
			\big( q_{t_l}^{x_l} \big)^2
			\big( q_{t_{r+1}}^{x_{r+1}} \big)^2.
\end{flalign}
We now search for a bound for $M_{t,r,l}^{t'} (x')$ when $t$ is sufficiently large.

\medskip 

\begin{claim}\label{190918185939}
There are constants $C, C', T > 0$ such that for all integers $t \geq T$ and for all $r, l, t', x'$ as above,
$$
	M_{t,r,l}^{t'} (x')
		\lesssim
		t^{- \xi/4} \
		(q_t^y)^2 \ C^l \
		t^{- \xi l (2d-5) / 4} \ 
		\exp \Big( C' (r-l) t^{\sublin + \xi - 1} \Big), 
$$
and, in the special case $l=r$, 
$$
M_{t,r,r}(y) \lesssim t^{- \xi /4} \ (q_t^y)^2 \ C^r \ t^{-\xi r (2d-5)/4}. 
$$
\end{claim}

\medskip 

\noindent We use Claim \ref{190918185939} to estimate $M_{t,r,l} (y)$ from~\eqref{190918154637} as follows:
\begin{flalign*}
	M_{t,r,l} (y)
		& \lesssim \alpha_d^{r-l} \ t^{- \xi/4} \ 
		(q_t^y)^2 \ C^l \
		t^{- \xi l (2d-5) / 4} \ 
		\exp \Big( C' (r-l) t^{\sublin + \xi - 1} \Big)
\\		& \leq t^{- \xi/4} \ (q_t^y)^2 \
			\Big( C \ t^{- \xi (2d-5) / 4} \Big)^l \
			\Big( \alpha_d \exp \big( C' t^{\sublin + \xi - 1} \big) \Big)^{r-l}.
\end{flalign*}
Then we combine this with~\eqref{eq:M_upper_bd} to obtain 
\begin{flalign*}
	M_{t,r} (y)
		& \lesssim t^{- \xi/4} (q_t^y)^2
			(r + 1) \sum_{l=1}^r \binom{r}{l} \Big( C t^{- \xi (2d-5) / 4} \Big)^l
			\Big( \alpha_d \exp \big( C' t^{\sublin + \xi - 1} \big) \Big)^{r-l}
\\		& = t^{- \xi/4} \ (q_t^y)^2 \
			(r+1) \ \Big( C t^{- \xi (2d-5) / 4} + \alpha_d \exp \big( C' t^{\sublin + \xi - 1} \big) \Big)^r.
\end{flalign*}
Finally, combining this estimate with~\eqref{eq:Jensen_B} and~\eqref{eq:two_gaps_1}, we obtain
$$
\left(\frac{1}{q^y_t} \left \langle \left \lvert B^{y,t}_2 \right \rvert \right \rangle \right)^2  
		\lesssim t^{-\xi/4} \ \sum_{r = 1}^{\infty} \
		(r + 1) \
		\lambda^r \ \Big( C t^{- \xi (2d-5) / 4} + \alpha_d \exp \big( C' t^{\sublin + \xi - 1} \big) \Big)^r.
$$
Since $d \geq 3$ and since $\xi < 1- \sublin$, one has 
$$
\lim_{t \to \infty} t^{\theta} \sup_{y:  \|y\|\leq t^{\sigma}, q^y_t > 0} \left(\frac{1}{q^y_t} \left \langle \left \lvert B^{y,t}_2 \right \rvert \right \rangle \right)^2 = 0 
$$
as long as $\theta < \xi / 4$ and hence~\eqref{190918141213} for all $\theta < \xi / 8$. 
To complete the proof of Part 2, it remains to prove Claim \ref{190918185939}.

\medskip

\paragraph*{\textit{Proof of Claim~\ref{190918185939}.}} 
By Lemma~\ref{lm:linear_functional}, there are constants $\rho_1, \rho_2 > 0$ such that for any $t \in \N$ and for any $y \in \Z^d$ with $\| y \| \leq \rho_1 t$ and $q^y_t > 0$, there is a linear functional $\varphi$ on $\R^d$ of norm $\| \varphi \| \leq \rho_2 \| y \| / t$ which satisfies 
\begin{equation}\label{eq:linear_functional}
	q_t^y e^{\varphi (y)}
		\gtrsim 
			t^{-d/2}
			\sum_{z \in \Z^d} q_t^z e^{\varphi (z)}. 
\end{equation}
Fix $t \in \N$ so large that $k(t) \geq 1$, as well as $t^{\sublin} \leq \rho_1 t$ and $\rho_2 t^{\sublin - 1} \leq 1$. Let $y \in \Z^d$ such that $\| y \| \leq t^{\sublin}$ and $q^y_t > 0$. The conditions $\rho_2 t^{\sublin - 1} \leq 1$ and $\| y \| \leq t^{\sublin}$ imply in particular that $\| \varphi \| \leq 1$ for the linear function $\varphi$ corresponding to $t$ and $y$. 
Let $t_1, \ldots, t_l, t_{r+1} \in \N$ and $x_1, \ldots, x_l, x_{r+1} \in \Z^d$ such that the conditions under the summation sign in~\eqref{190918155153} hold. In the special case $l=r$, replace $t'$ and $x'$ with $t$ and $y$, respectively, here and in the remainder of the proof. By Lemma~\ref{lm:lin_func_upper_bd}, there is a constant $c_1 > 0$ such that
\begin{flalign*}  
	q_{t_1}^{x_1} \cdots q_{t_l}^{x_l} q_{t_{r+1}}^{x_{r+1}} 
		&= 
		e^{\varphi(-x')} 
		\prod_{j \in \{1, \ldots, l, r+1\}} 
			e^{\varphi(x_j)} q_{t_j}^{x_j}
\\		&\leq
		e^{\varphi(-x')}
		\prod_{j \in \{1, \ldots, l, r+1\}} 
		\biggl( c_1 t_j^{-d/2} 
		\sum_{z \in \Z^d} q_{t_j}^z e^{\varphi(z)} \biggr)
\\		&=
		e^{\varphi(-x')}
		\sum_{z \in \Z^d} q_{t'}^z e^{\varphi(z)}
		\prod_{j \in \{1, \ldots, l, r+1\}} 
		\biggl( c_1 t_j^{-d/2} \biggr)
\\		&=
		e^{\varphi(-x')} \
		\Phi (0)^{t'}
		\prod_{j \in \{1, \ldots, l, r+1\}} 
		\biggl( c_1 t_j^{-d/2} \biggr),
\end{flalign*}
where in the third line we used the fact that $\varphi$ is a linear functional, and in the fourth line we used~\eqref{190918172610}, where $\Phi$ was defined in~\eqref{eq:phi_def}. 
Since $t' < t$ and $\Phi (0) \geq 1$, it follows from~\eqref{eq:linear_functional} that $\Phi (0)^{t'} \leq \Phi (0)^t \lesssim t^{d/2} q_t^y e^{\varphi (y)}$.
As a result, for all positive integers $t_1, \ldots, t_l, t_{r+1}$ such that $t_1 + \ldots + t_l + t_{r+1} = t'$ and $t_{r+1} \geq t_1, \ldots, t_l \geq t^{\xi}$, one has 
$$
	\max_{x_1 + \ldots + x_l + x_{r+1} = x'}  
		q_{t_1}^{x_1} \cdots q_{t_l}^{x_l} q_{t_{r+1}}^{x_{r+1}}
			\lesssim 
		t^{d/2} \ q_t^y \ e^{\varphi(y - x')}
		\prod_{j \in \{1, \ldots, l, r+1\}}
			\biggl(c_1 t_j^{-d/2} \biggr).
$$
Furthermore, the sum $\sum q_{t_1}^{x_1} \cdots q_{t_l}^{x_l} q_{t_{r+1}}^{x_{r+1}}$ over all tuples $(x_1, \ldots, x_l, x_{r+1})$ such that $x_1 + \cdots x_l + x_{r+1} = x'$ equals $q_{t'}^{x'}$, and by Lemma~\ref{lm:ratio_estimate} there are constants $c, \rho > 0$ such that 
$$ 
	q_{t'}^{x'} 
		\leq 
		q_t^y 
			\big(1 + O(t^{-2/5}) \big)
			\exp \Bigg( \frac{c}{t}
				\Bigg(
					\|y\| \| y - x' \| + \|y\| (t-t')
						+ \ln (t) (t-t')
				\Bigg)
				\Bigg),
$$
for $t$ so large that $t^{\sublin} \leq \rho t$.
Therefore,
\begin{flalign*}
\label{eq:est_sum_squares_7}
	\sum_{\mathclap{x_1 + \ldots + x_l + x_{r+1}=x'}}
		\;
		\big( q_{t_1}^{x_1} \big)^2
				\cdots
			\big( q_{t_l}^{x_l} \big)^2
			\big( q_{t_{r+1}}^{x_{r+1}} \big)^2
		&\lesssim
			t^{d/2} \ q_t^y \ q_{t'}^{x'} e^{\varphi (y - x')}
		\;\;\;\;
			\prod_{\mathclap{j \in \{1, \ldots, l, r+1\}}}
		\qquad
				\biggl(c_1 t_j^{-d/2} \biggr)
\\		&\lesssim t^{d/2} \ (q_t^y)^2 \
	P (t) \qquad
	\prod_{\mathclap{j \in \{1, \ldots, l, r+1\}}}
		\qquad
				\biggl(c_1 t_j^{-d/2} \biggr),
\end{flalign*}
where $P (t) \coleq \exp \bigg( \dfrac{c'}{t} \Big(
					2 \|y\|  \| y - x' \| + \|y\| (t-t')
						+ \ln (t) (t-t')
				\Big)
				\bigg)$, for a constant $c' > 0$. In the second line of the estimate above, we also used that $\| \varphi \| \leq \rho_2 \|y\| /t$. 
				
%===================================
%\begin{lemma}             \label{lm:lemma_3}
%There is $c > 0$ such that for any $n \in \N$, $l \in \N_0$, and $M > 0$,  
%\begin{equation}    \label{eq:lemma_3} 
%\sum_{\substack{n_1 + \ldots + n_{l+1} = n, \\ n_1, \ldots, n_{l+1} \geq M}} \prod_{j=1}^{l+1} n_j^{-\frac{d}{2}} \leq \frac{c^l}{M^{l(\frac{d}{2}-1)}} n^{-\frac{d}{2}}.
%\end{equation}
%Here, $n_1, \ldots, n_{l+1}$ are always integers. 
%\end{lemma}
%\noindent We prove Lemma~\ref{lm:lemma_3} in Subsection~\ref{appendix_calc}.
%=================================	

Together with Lemma~\ref{lm:lemma_3} from the appendix, we obtain the following estimate on $M^{t'}_{t,r,l} (x')$:
\begin{flalign*}
%\label{eq:final_est_1}
	M^{t'}_{t,r,l} (x')
		&\lesssim t^{d/2} \ (q_t^y)^2 \
		P (t) \qquad
		\sum_{\mathclap{ \substack{
%			n_1, \ldots, n_l, n_{r+1} \in \N
		\\	t_1 + \ldots + t_l + t_{r+1} = t'
		\\	t_{r+1} \geq t_1, \ldots, t_l \geq t^\xi
			}}} \qquad \quad \;
		\prod_{j \in \{1, \ldots, l, r+1\}}
				\biggl(c_1 t_j^{-d/2} \biggr)
\\
		&\lesssim (q_t^y)^2 \ C^l \
			t^{- \xi l (d-2) / 2}
			\left( \frac{t}{t'} \right)^{d/2} P(t),
\\		&\leq
			t^{- \xi/4} \
				(q_t^y)^2 \ C^l \
				t^{- \xi l (2d-5) / 4}
				\left( \frac{t}{t'} \right)^{d/2} P(t),
%\notag
\end{flalign*}
where $C > 0$ is a constant. 
It remains to bound $(t / t')^{d/2} P(t)$.
We estimate the following expressions involved in $(t / t')^{d/2} P(t)$ like so:
\begin{gather*}
	\left( \frac{t}{t'} \right)^{d/2} 
		\leq \exp \left( \frac{d}{2} \ln (t) \frac{t - t'}{t-1} \right),
\qquad\quad
	t - t' = \sum_{j = l+1}^r t_j \leq (r-l) t^\xi,
\\	\| y - x' \| \leq \sum_{j = l+1}^r \| x_j \| \leq \sum_{j = l+1}^r t_j \leq (r-l) t^\xi,
\end{gather*}
where $\sum_{j=l+1}^r \| x_j \| \leq \sum_{j=l+1}^r t_j$ is valid under the assumption that $q_{t_1}^{x_1} \ldots q_{t_{r+1}}^{x_{r+1}} > 0$.
Then, using $\|y\| \leq t^{\sublin}$, we obtain
$$
	\left( \frac{t}{t'} \right)^{d/2} P(t) 
		\leq \exp \big( C' (r-l) t^{\sublin + \xi - 1} \big)
$$
for some constant $C' > 0$. This completes the proof of Claim~\ref{190918185939}.

%======================================
%======================================
\subsection{Proof of the Central Lemma, Part 3: The main contribution}
\label{ssec:proof_central_lemma_part3}
%======================================
%======================================
Let $t$ be so large that $k(t) \geq 1$ and let $y \in \Z^d$ such that $\| y \| \leq t^{\sublin}$ and $q^y_t > 0$. For $\ibf \in I_1(t,r,m)$ and $\zbf \in (\Z^d)^r$, define
%===
\begin{equation} \label{eq:def_qminusm}
q_{t,\hat{m}}^y(\ibf,\zbf) := q_{i_1}^{z_1} \ldots  \widehat{ q_{i_{m} - i_{m-1}}^{z_{m}-z_{m-1}} }\ldots q_{t-i_r}^{y-z_r},
\end{equation}
where the factor with the hat is absent; in other words, we remove the transition probability corresponding to the huge gap.

Now decompose $B_3^{y,t}$ further, depending on the position of the huge gap 
1) at the begining, 2) in the middle, or 3) at the end, as follows:

\begin{equation*}
B_3^{y,t} = q_{t}^y \sum_{i=1}^3 \left(F_i^{y,t} + L_i^{y,t} \right),
\end{equation*}
where
%===
\begin{align}       \label{eq:def_F} 
	F_{1}^{y,t} \coleq & \sum_{1 \leq r \leq k(t)}  \sum_{\ibf \in I_1(t, r,1), \zbf} 
	q_{t,\hat{1}}^y(\ibf, \zbf)  
	\prod_{j=1}^r h(z_j, i_j), \\ 
	F_2^{y,t} \coleq & \sum_{2 \leq r \leq k(t)} \sum_{m=2}^r \; 
	\sum_{\ibf \in I_1(t,r,m),  \zbf }  
	q_{t,\hat{m}}^y(\ibf, \zbf)  
	\prod_{j=1}^r h(z_j, i_j), \notag \\ 
	F_3^{y,t} \coleq & \sum_{1 \leq r \leq k(t)} \; 
	\sum_{\ibf \in I_1(t,r,r+1),  \zbf } 
	q_{t,\widehat{r+1}}^y(\ibf, \zbf)  
	\prod_{j=1}^r h(z_j, i_j);    \notag 
\end{align}
%===
and the error terms are given by
\begin{align*}
	L_1^{y,t} \coleq & \sum_{1 \leq r \leq k(t)} \;
	\sum_{\ibf \in I_1(t,r,1), \zbf}
	\frac{q_{i_1}^{z_1} - q_{t}^{y}}{q^{y}_{t}} 
	q_{t,\hat{1}}^y(\ibf, \zbf)  
	\prod_{j=1}^r h(z_j, i_j), \\
	L_2^{y,t} \coleq & \sum_{2 \leq r \leq k(t)} \sum_{m=2}^r \;
	\sum_{\ibf \in I_1(t,r,m), \zbf} 
	\frac{q_{i_m - i_{m-1}}^{z_m - z_{m-1}} - q_{t}^{y}}{q_{t}^{y}} 
	q_{t,\hat{m}}^y(\ibf, \zbf)     
	\prod_{j=1}^r h(z_j, i_j),   \\  
	L_3^{y,t} \coleq & \sum_{1 \leq r \leq k(t)} \;
	\sum_{\ibf \in I_1(t,r,r+1), \zbf} 
	\frac{q_{t-i_r}^{y-z_r}- q_{t}^{y}}{q^{y}_{t}} 
	q_{t,\widehat{r+1}}^y(\ibf, \zbf) 
	\prod_{j=1}^r h(z_j, i_j). 
\end{align*}

\bigskip
\noindent 
Notice that $F_1^{y,t}, F_2^{y,t}, F_3^{y,t}$ are well-defined even if $q^y_t = 0$. We first show that the contribution from each error term is negligible.

\begin{lemma} \label{lm:L_convergence}
There is $\theta > 0$ such that
\begin{equation*} 
\lim_{t \to \infty} t^{\theta} \sup_{y: \|y\| \leq t^{\sublin}, q^y_t > 0}  \biggl \langle \biggl \lvert  \sum_{i=1}^3 L_i^{y,t} \biggr \rvert \biggr \rangle  = 0. 
\end{equation*}
\end{lemma}

%======================================
%======================================
%\subsubsection{Proof of the Main Claim, Part 1: Convergence of $L$-terms}
%\label{ssec:L_terms}
%======================================
%======================================

\begin{proof}
It is enough to show that there is $\theta > 0$ such that for $i \in \{1,2,3\}$, 
%===
\begin{equation} \label{eq:D_seq_conv}
\lim_{t \to \infty} t^{\theta} \sup_{y: \|y\| \leq t^{\sublin}, q^y_t > 0} \left \langle \left( L_i^{y,t} \right)^2 \right \rangle = 0. 
\end{equation}
%===

For $t$ so large that $k(t) \geq 1$ and for $y \in \Z^d$ such that $\| y \| \leq t^{\sublin}$ and $q^y_t > 0$, one has 
$$
	\left \langle \left( L_i^{y,t} \right)^2 \right \rangle  = 
	\sum_{1 \leq r \leq k(t)} 
	\lambda^r a_i(r)
	\sum_{\substack{\tbf \in \Xi_r^i, \xbf}} 
%		\sum_{z_1, \ldots, z_r \in \Z^d} 
	 \left(q_{t_1}^{x_1} \right)^2 \ldots \left(q_{t_r}^{x_r} \right)^2 \frac{(q_{t-t_1 - \ldots - t_r}^{y-x_1 -\ldots -x_r} - q_t^y)^2}{(q_t^y)^2}, 
$$
where $a_i(r) \coleq 1$ if $i=1,3$, $a_2(r) \coleq (r-1) \id_{r \geq 2}$,\\ 
\begin{align*}
\Xi_r^i \coleq& \left\{\tbf = (t_1,\ldots, t_r) \in \N_0^r ~:~
			\substack{ \displaystyle  t_1 \geq 0, t_2,\ldots, t_r > 0 
			\\ \displaystyle
			t_1 + \cdots + t_r  \leq r t^\xi }\right\}, \quad i=1,3, \\
\Xi_r^2 \coleq& \left\{\tbf = (t_1,\ldots, t_r) \in \N_0^r  
~:~
			\substack{ \displaystyle  t_1, t_r \geq 0, t_2, \ldots, t_{r-1} > 0 
			\\ \displaystyle
			t_1 + \cdots + t_r  \leq r t^\xi } \right\}, 
\end{align*}
and where the sum $\sum_{\xbf}$ is taken over all $\xbf = (x_1, \ldots, x_r) \in (\Z^d)^r$. 
%
%
%\coleq
%		\left\{ (i_1, \ldots, i_r) ~:~
%			\substack{ \displaystyle 
%				0 \leq i_1 < \ldots < i_r \leq n 
%			\\	\displaystyle
%				\text{the gap between $i_{m-1}$ and $i_m$ is huge}
%			}
%			\right\}.
%			
%			

%n_1 \geq 0, (n_2, \ldots, n_r) \in \N^{r-1}, \\ n_1 + \ldots + n_r < r n^{\xi}

%$$
%	D_n^{\vartheta}(1,y) 
%		\coleq
%		\begin{cases}
%			\displaystyle \bigg.
%				~ 0 & \quad \text{if $q^y_n = 0$,}
%	\\		\displaystyle \sum_{1 \leq r \leq k(n)} \vartheta^r 
%	\sum_{\substack{n_1 \geq 0, (n_2, \ldots, n_r) \in \N^{r-1}, \\ n_1 + \ldots + n_r < r n^{\xi} \\ \zbf}} 
%%		\sum_{z_1, \ldots, z_r \in \Z^d} 
%	 \left(q_{n_1}^{z_1} \right)^2 \ldots \left(q_{n_r}^{z_r} \right)^2 \frac{(q_{n-n_1 - \ldots - n_r}^{y-z_1 -\ldots -z_r} - q_n^y)^2}{(q_n^y)^2}
%				& \quad \text{if $q^y_n > 0$.}
%		\end{cases}
%$$
%
%and
%
%$$
%	D_n^{\vartheta}(2,y) 
%		\coleq
%		\begin{cases}
%			\displaystyle \bigg.
%				~ 0 & \quad \text{if $q^y_n = 0$,}
%	\\		\displaystyle \sum_{2 \leq r \leq k(n)} \vartheta^r (r - 1)  \sum_{\substack{n_1, n_r \geq 0, \\ (n_2, \ldots, n_{r-1}) \in \N^{r-2} \\ n_1+ \ldots + n_r < r n^{\xi} \\ \zbf}}
%	 \left( q_{n_1}^{z_1} \right)^2 \ldots \left( q_{n_r}^{z_r} \right)^2 	
%	\frac{\left( q_{n - n_1 - \ldots - n_r}^{y - z_1 - \ldots - z_r} - q_n^y \right)^2}{(q^y_n)^2}.  
%					& \quad \text{if $q^y_n > 0$.}
%		\end{cases}
%$$
\medskip
The convergence in \eqref{eq:D_seq_conv} relies on $q_{t-t_1 -\ldots -t_r}^{y-x_1 -\ldots - x_r}$ being close to $q_t^y$ in the following sense: Let $\rho > 0$ be the constant from Lemma~\ref{lm:ratio_estimate}, and assume that $t$ is so large that $t^{\sublin} \leq \rho t$. Let $1 \leq r \leq k(t)$, $t_1, t_r \in \N_0$, $t_2, \ldots, t_{r-1} \in \N$ with $t_1 + \ldots + t_r \leq r t^{\xi}$. Without loss of generality, let $x_1, \ldots, x_r \in \Z^d$ such that $q_{t_1}^{x_1} \ldots q_{t_r}^{x_r} > 0$, as otherwise the contribution to $\langle (L_i^{y,t})^2 \rangle$ is zero. 

%===
\begin{claim} \label{cl:close_qs}
There is a constant $c_3 > 0$ such that 
\begin{equation*}
\left(\frac{q_{t-t_1-\ldots-t_r}^{y-x_1-\ldots-x_r} - q_t^y}{q_t^y} \right)^2 \leq \left(1 + O(t^{-\frac{2}{5}}) \right)  \exp \left(c_3  r  t^{\sublin+\xi -1} \right) - 1.
\end{equation*}
\end{claim}
%===

\noindent Using this claim we can bound $\sup_{y: \|y\| \leq t^{\sublin}, q^y_t > 0} \langle (L_i^{y,t})^2 \rangle$ as follows:
%===
\begin{align}       \label{eq:sup_error_bound} 
\sup_{y:\|y\| \leq t^{\sublin}, q^y_t > 0} \left \langle \left(L_i^{y,t} \right)^2 \right \rangle  
\lesssim &  \sum_{r=1}^{\infty} \lambda^r r
\sum_{(t_1, \ldots, t_r) \in \N^{r}, \xbf} %\sum_{z_1, \ldots, z_r \in \Z^d} 
\left( q_{t_1}^{x_1} \right)^2 \ldots \left( q_{t_r}^{x_r} \right)^2  \\
& \left(\left(1 + O(t^{-\frac{2}{5}}) \right) \exp \left(c_3 r t^{\sublin+\xi -1} \right) - 1 \right) \notag \\
\lesssim& \sum_{r=1}^{\infty} \left(\alpha_d \lambda \right)^r r \left(\left(1 + O(t^{-\frac{2}{5}}) \right) \exp\left(c_3 r t^{\sublin+\xi -1} \right) - 1 \right).   \notag 
\end{align}
%===

Let $\theta \in (0, \min\{2/5; 1 - \sigma - \xi\})$. The definition of the Landau symbol $O(t^{-\frac{2}{5}})$ implies that there are constants $C, T > 0$ such that for $t > T$, 
\begin{align*}    
t^{\theta} & \left( \left(1 + O(t^{-\frac{2}{5}}) \right)  \exp \left(c_3 r t^{\sigma + \xi - 1} \right) - 1 \right) \\
&\leq c_3 r t^{\sigma + \xi - 1 + \theta} \ \frac{\exp(c_3 r t^{\sigma + \xi - 1})-1}{c_3 r t^{\sigma + \xi - 1}} + C t^{-\frac{2}{5} + \theta} \exp \left(c_3 r t^{\sigma + \xi - 1} \right)  \\
&\leq \left(c_3 r + C \right) \exp \left(c_3 r t^{\sigma + \xi - 1} \right), 
\end{align*}
where, in the third line, we used that $(e^x - 1)/x \leq e^x$ for $x > 0$.  
Hence, 
\begin{align*}
t^{\theta} \sum_{r=1}^{\infty} \ \biggl \lvert \left(\alpha_d \lambda \right)^r r & \left( \left(1 + O(t^{-\frac{2}{5}}) \right) \exp \left(c_3 r t^{\sigma + \xi - 1} \right) - 1 \right)  \biggr \rvert \\
& \leq \sum_{r=1}^{\infty} \left(\alpha_d \lambda \exp \left(c_3 t^{\sigma + \xi - 1} \right) \right)^r \left(c_3 r^2 + C r \right). 
\end{align*}
For $\varphi \in (\alpha_d \lambda, 1)$ and $t$ so large that $\alpha_d \lambda \exp(c_3 t^{\sigma + \xi -1}) \leq \varphi$, the series on the right is dominated by the convergent series $\sum \varphi^r (c_3 r^2 + Cr)$. Dominated convergence and~\eqref{eq:sup_error_bound} then imply~\eqref{eq:D_seq_conv} for $\theta \in (0, \max\{2/5; 1 - \sigma - \xi\})$.

To complete the proof of Lemma~\ref{lm:L_convergence}, it remains to prove Claim~\ref{cl:close_qs}.
\smallskip

\paragraph*{\textit{Proof of Claim~\ref{cl:close_qs}.}} 
Let $x' \coleq x_1+\cdots + x_r$ and $t' \coleq t_1 + \cdots + t_r$. Observe that  $q_{t - t'}^{y - x'} > 0$: Indeed, notice first that 
$t - t' \geq t - k(t) t^{\xi}$ since $t' \leq r t^\xi$. As $k(t)$ is of order $t^{\kappa_1}$ and $\kappa_1 + \xi < 1$, the term $t - t'$ is of order $t$. Moreover,
$$
\| y - x' \| \leq t^{\sublin} + \sum_{j=1}^r \|x_j\| \leq t^{\sublin} + \sum_{j=1}^r t_j \leq t^{\sublin} + k(t) t^{\xi}, 
$$
which is of smaller order than $t - t'$. Finally, $t - t'$ and $\| y - x' \|_1$ have the same parity because $q^y_t > 0$ and $q_{t_1}^{x_1} \ldots q_{t_r}^{x_r} > 0$. 
Now, we derive an upper bound on $\lvert q_{t-t'}^{y-x'} - q_t^y \rvert/{q_t^y}$. If $q_{t-t'}^{y-x'} \geq q_t^y$, then combining Lemma~\ref{lm:ratio_estimate} with the estimate $\| x' \| \leq t'  \leq r t^{\xi}$ gives  
%===
\begin{align} 
\label{eq:qn'_1}
\frac{\lvert q_{t-t'}^{y-x'} - q_t^y \rvert}{q_t^y} \leq& \left(1+O(t^{-\frac{2}{5}}) \right) \exp \left(c \left(2 t^{\sublin-1} r t^{\xi} + \frac{\ln(t)}{t} r t^{\xi} \right) \right) - 1 
\\ \notag
\leq& \left(1+O(t^{-\frac{2}{5}}) \right) \exp \left(c_1 r t^{\sublin+\xi -1} \right) - 1
\end{align}
%===
for some constant $c_1 > 0$. 
If $q_t^y > q_{t-t'}^{y-x'}$, we argue as follows:
Let $t \in \N$ be so large that  
$
t^{\sublin} + k(t) t^{\xi} \leq \rho (t-t'). 
$
Then,  
$$
\|y-x'\| \leq t^{\sublin} + k(t) t^{\xi} \leq \rho (t-t')
$$ 
and Lemma~\ref{lm:ratio_estimate} with the estimate $k(t) t^{\xi} \lesssim t^{\kappa_1 + \xi} \leq t^{\sigma}$ (coming from $\xi < \kappa_2 - \kappa_1 < \sigma - \kappa_1$) yields  
%===
\begin{align} \label{eq:qn'_2}
& \frac{\lvert q_{t-t'}^{y-x'} - q_t^y \rvert}{q_t^y} \leq \frac{q_t^y}{q_{t-t'}^{y-x'}} - 1 \\
\leq&  \left(1 + O(t^{-\frac{2}{5}}) \right) \exp \left( c \left( \frac{t^{\sublin} + k(t) t^{\xi}}{t-t'} 2 r t^{\xi} + \ln(t-t') \frac{r t^{\xi}}{t-t'} \right) \right)  -1 \notag \\ 
\leq& \left(1 + O(t^{-\frac{2}{5}}) \right) \exp \left(c_2  r t^{\sublin +\xi-1} \right) -1.   \notag 
\end{align}
%===
Using the general fact that $(a-1)^2 \leq a^2 - 1$ for every $a \geq 1$, in either case \eqref{eq:qn'_1} or \eqref{eq:qn'_2}, we have the following bound:
$$ 
\left(\frac{q_{t-t'}^{y-x'} - q_t^y}{q_t^y} \right)^2 \leq \left(1 + O(t^{-\frac{2}{5}}) \right) \exp \left(c_3 r t^{\sublin+\xi -1} \right) - 1,  
$$ 
where $c_3 > 0$ is a constant.

\end{proof}

%======================================
%======================================

In order to deal with the $F_i$'s defined in~\eqref{eq:def_F}, the strategy is to first define suitable truncations of the partition functions. Fix $\xi_1, \xi_2$ satisfying 
$$
0 < \xi_1 < \xi_2 < \xi, 
$$
and notice that since $\xi + \sigma < 1$, we have $\xi_1 + \sigma < 1$. Now set
$$
T_{0,0}^t := 1 + \sum_{1 \leq r \leq t^{\xi_1} +1} \sum_{\substack{\ibf \in I_{r,t}, \zbf \\ i_r \leq t^{\xi_2}}} q^y_{t, \widehat{r+1}}(\ibf, \zbf) \prod_{j=1}^r h(z_j, i_j)
$$
and 
$$
T_0^{y,t} := 1 + \sum_{1 \leq r \leq t^{\xi_1} + 1} \sum_{\substack{\ibf \in I_{r,t}, \zbf \\ t - t^{\xi_2} \leq i_1}} q^y_{t, \hat{1}}(\ibf, \zbf) \prod_{j=1}^r h(z_j, i_j),  
$$
where $q^y_{t, \widehat{r+1}}(\ibf, \zbf)$ and $q^y_{t, \hat{1}}(\ibf, \zbf)$ are defined according to~\eqref{eq:def_qminusm}, with $q^y_{t, \widehat{r+1}}(\ibf, \zbf)$ not depending on $y$. Notice that $T_{0,0}^t$ and $T_0^{y,t}$ are truncations of the partition functions $Z_{0,0}^t$ and $Z_0^{y,t}$, respectively (see~\eqref{eq:Z_M} and~\eqref{eq:Z_backwards}). The convergence statement in~\eqref{eq:convergence_3} will follow from the lemmas below.

\begin{lemma} \label{lm:main_claim1-4}
There is $\theta > 0$ such that
\begin{align}
	\lim_{t \to \infty} t^{\theta} \sup_{\|y\| \leq t^{\sublin}} \left \langle \left \lvert F_2^{y,t} - (T_{0,0}^t - 1) (T^{y,t}_0 - 1) \right \rvert \right \rangle =& 0, \label{eq:conv_4} \\
	\lim_{t \to \infty} t^{\theta} \sup_{\|y\| \leq t^{\sublin}} \left \langle \left \lvert F_1^{y,t} - (T^{y,t}_0 - 1) \right \rvert \right \rangle =& 0, \label{eq:conv_2} \\
	\lim_{t \to \infty} t^{\theta} \sup_{\|y\| \leq t^{\sublin}} \left \langle \left \lvert F_3^{y,t} - (T_{0,0}^t - 1) \right \rvert \right \rangle =& 0. \label{eq:conv_3} 
\end{align}
\end{lemma}

\begin{lemma} \label{lm:main_claim5}
There is $\theta > 0$ such that
\begin{align}
\lim_{t \to \infty} t^{\theta} \sup_{\|y\| \leq t^{\sublin}} \left \langle \left \lvert T^{y,t}_0 T_{0,0}^t - Z_{0,0}^{\infty} Z_{-\infty}^{y,t} \right \rvert \right \rangle =& 0.   \label{eq:conv_5}
\end{align}
\end{lemma}

The convergence statement in~\eqref{eq:conv_4} is shown in Section~\ref{ssec:convergence_middle}. We show the convergence statements in~\eqref{eq:conv_2} and~\eqref{eq:conv_3} in Section~\ref{ssec:gap_end}, and the one in~\eqref{eq:conv_5} in Section~\ref{ssec:convergence_ergodicity}.

%======================================
%======================================
\section{Main Contribution: Proofs of Lemmas~\ref{lm:main_claim1-4} and~\ref{lm:main_claim5}}
\label{ssec:main_contribution_lemmas}
%======================================
%======================================

%======================================
%======================================
%\label{ssec:L_terms}
%======================================
%======================================

%======================================
%======================================
\subsection{Proof of Lemma~\ref{lm:main_claim1-4}, ~\eqref{eq:conv_4}: Convergence for one huge gap in the middle}
\label{ssec:convergence_middle}
%======================================
%======================================

One has    
%===
	\begin{align}   \label{eq:T_product_1}
(T_{0,0}^t - 1) (T^{y,t}_0 - 1)   
	=& \sum_{1 \leq r \leq t^{\xi_1}+1} \sum_{1 \leq s \leq t^{\xi_1}+1} \sum_{\substack{0 \leq i_1 < \ldots < i_r \leq t^{\xi_2}, \\ z_1, \ldots, z_r \in \Z^d}} \sum_{\substack{t - t^{\xi_2} \leq l_1 < \ldots < l_s \leq t, \\ c_1, \ldots, c_s \in \Z^d}}  \\
	&  q_{i_1}^{z_1} \ldots q_{i_r - i_{r-1}}^{z_r - z_{r-1}} q_{l_2 - l_1}^{c_2 - c_1} \ldots  q_{t - l_s}^{y - c_s} \prod_{j=1}^r h(z_j, i_j) \prod_{k=1}^s h(c_k, l_k).  \notag 
\end{align}
%===
%
Define the set
$$
	V(t,r,m)
		\coleq
		\left\{ \ibf = (i_1, \ldots, i_r) \in I_1(t,r,m) ~:~
			\substack{ \displaystyle 
				0 \leq i_1 < \ldots < i_{m-1} \leq t^{\xi_2} 
			\\	\displaystyle
				\quad t - t^{\xi_2} \leq i_m < \ldots < i_r \leq t
			}
			\right\}
$$
and its complement in $I_1(t,r,m)$
$$
	W(t,r,m)
		\coleq
		\left\{ \ibf = (i_1, \ldots, i_r) \in I_1(t,r,m) ~:~
			\substack{ \displaystyle 
				i_{m-1} > t^{\xi_2} \quad \text{or} \quad i_m < t - t^{\xi_2} 
			}
			\right\}.
$$
Recall the notation $q^{y}_{t,\hat m}(\ibf,\zbf)$ from \eqref{eq:def_qminusm}. %===
Making the change of summation indices $r\coleq r+s$ and $m\coleq r+1$ in ~\eqref{eq:T_product_1}, one has %===
\begin{align}    \label{eq:p6_10}
 (T_{0,0}^t - 1) (T^{y,t}_0 - 1)  = \
&\hspace{-3mm}  \sum_{2 \leq r \leq t^{\xi_1}+ 2} \ \sum_{m=2}^r \ 
\sum_{\ibf \in V(t,r,m), \zbf}  
q_{t,\hat m}^y(\ibf,\zbf) \prod_{j=1}^r h(z_j, i_j) \\
+& \sum_{\substack{t^{\xi_1}+ 2  < r, \\ r \leq 2 t^{\xi_1}+ 2}} \
 \sum_{\substack{r - t^{\xi_1} \leq m, \\ m \leq t^{\xi_1} + 2}} \ 
\sum_{\ibf \in V(t,r,m), \zbf} 
q_{t,\hat m}^y(\ibf,\zbf) \prod_{j=1}^r h(z_j, i_j).   \notag
	\end{align}
The identity in~\eqref{eq:p6_10} allows us to rewrite $F_2^{y,t} - (T_{0,0}^t - 1) (T_0^{y,t} - 1)$ as
$
f^{y,t}_{2;1} + f^{y,t}_{2;2} + f^{y,t}_{2;3}, 
$
where 
	\begin{align*}
	f^{y,t}_{2;1} \coleq & 
	\sum_{r \in R^1} \ 
	\sum_{m = 2}^r  \ 
	\sum_{\ibf \in W(t,r,m), \zbf} 
	q^{y}_{t,\hat m}(\ibf,\zbf) 
	\prod_{j=1}^r h(z_j, i_j), \\
	f^{y,t}_{2;2} \coleq & 
	\sum_{r \in R^2} \biggl(
 	\sum_{m=2}^r 
	\sum_{\ibf \in I_1(t,r,m),  \zbf }  q^y_{t, \hat{m}}(\ibf, \zbf) \prod_{j=1}^r h(z_j, i_j) \\
	& -
	 \sum_{\substack{r - t^{\xi_1} \leq m, \\ m \leq t^{\xi_1}+ 2}} \ 
\sum_{\ibf \in V(t,r,m), \zbf} 
q_{t,\hat m}^y(\ibf,\zbf) \prod_{j=1}^r h(z_j, i_j) \biggr),	\\
	f^{y,t}_{2;3} \coleq & 
	\sum_{r \in R^3} \ 
	\sum_{m=2}^r \ 
	\sum_{\ibf \in I_1(t,r,m), \zbf} 
	q^{y}_{t,\hat m}(\ibf,\zbf) 
	\prod_{j=1}^r h(z_j, i_j),
	\end{align*}

\noindent $R^1 \coleq \{r \in \N: \ 2 \leq r \leq t^{\xi_1} + 2\}$, $R^2 \coleq \{r \in \N: \ t^{\xi_1} + 2 < r \leq 2 t^{\xi_1} + 2\}$, and $R^3 \coleq \{r \in \N: \ 2 t^{\xi_1} + 2 < r \leq k(t)\}$.

\bigskip
	
In order to prove~\eqref{eq:conv_4}, it is then enough to show existence of $\theta > 0$ such that for $i=1,2,3$,
\begin{equation}  \label{eq:1_i_7}
\lim_{t \to \infty} t^{\theta} \sup_{\|y\| \leq t^{\sublin}} \left \langle \left(f_{2;i}^{y,t} \right)^2 \right \rangle = 0. 
\end{equation}
For $i = 1, 3$, one has  
\begin{equation}     \label{eq:f2_i_expansion}
\langle (f_{2;i}^{y,t})^2  \rangle = \sum_{r \in R^i} \lambda^r \sum_{m=2}^r  \ \sum_{\ibf \in H^i(t,r,m), \zbf} 
q^{y}_{t,\hat m}(\ibf,\zbf)^2, 
\end{equation}
%===
where $H^1(t,r,m) \coleq W(t,r,m)$ and $H^3(t,r,m) \coleq I_1(t,r,m)$. Notice furthermore that $\langle (f_{2;2}^{y,t})^2  \rangle$ is bounded by~\eqref{eq:f2_i_expansion} with $i=2$ and $H^2(t,r,m) \coleq I_1(t,r,m)$.
%===
Now, we take up cases $i=1,2,3$ separately.

\smallskip 

\paragraph*{\textsc{Case} $i = 1$.} Since for $\ibf \in W(t,r,m)$, 
\begin{align*}
i_1 + (i_2 - i_1) + \ldots + (i_{m-1} - i_{m-2}) +& (i_{m+1} - i_m) + \ldots + (t - i_r) \\
=& i_{m-1} - i_m + t \geq \max\{i_{m-1}; t-i_m\} > t^{\xi_2}, 
\end{align*}
the expression in~\eqref{eq:f2_i_expansion} is dominated by 
%===
$$
\sum_{r \in R^1} r  \lambda^r 
\sum_{\substack{t_1, \ldots, t_r \in \N, \xbf \\ t_1+ \ldots + t_r > t^{\xi_2}}} 
\left(q_{t_1}^{x_1} \right)^2 \ldots \left(q_{t_r}^{x_r} \right)^2 
\lesssim \sum_{r=1}^{\infty} r^2 (\alpha_d \lambda)^{r} \sum_{j > \frac{t^{\xi_2}}{t^{\xi_1}+2}} \frac{1}{j^{\frac{d}{2}}} \lesssim t^{(\xi_2 - \xi_1) (1-\frac{d}{2})}. 
$$
This implies~\eqref{eq:1_i_7} for $\theta < (\xi_2-\xi_1) (\tfrac{d}{2}-1)$. 

\smallskip

\paragraph*{\textsc{Case} $i=2$.} The expression in~\eqref{eq:f2_i_expansion} is dominated by  
%===
$$
		\sum_{r \in R^2} r \lambda^r  
		\sum_{t_1, \ldots, t_{r} \in \N, \xbf} 
			\left(q_{t_1}^{x_1} \right)^2 \ldots \left(q_{t_{r}}^{x_{r}} \right)^2 
				\leq \sum_{r > t^{\xi_1}} r (\alpha_d \lambda)^r.
$$%=== 
From this estimate we deduce~\eqref{eq:1_i_7} for all $\theta > 0$. 

\smallskip 

%=============
\paragraph*{\textsc{Case} $i = 3$.} The expression in~\eqref{eq:f2_i_expansion} is dominated by 
%=== 
\begin{equation*}
\sum_{r \in R^3} r (\alpha_d \lambda)^r, 
\end{equation*}
%===
%
which converges to $0$ as $t\to\infty$ faster than any polynomial by the same argument as in the case $i=2$. %===
%

%%===
%\begin{align}   \label{eq:R_sums_3}
%	& \sup_{\|y\| \leq t^{\sublin}} \frac{1}{p_t^y} \
%	e^{-\beta^2t^{\xi_1}}
%	\sum_{l_{\bullet} \in J(t-2t^{\xi_1})} 
%	q^y_{\iota(y,{l_\bullet})} 
%	\Pp(l_{\bullet})  \\
%	& \sum_{r_\bullet\in \N_0} (r_\bullet+1) (\alpha \psi)^{r_\bullet} \\
%	&e^{\beta^2t^{\xi_1}} \sum_{\substack{l_-, l_+ \in J(t^{\xi_1}), \\ l_-  \leq l_+}} 
%	\sum_{\substack{r_- \leq l_-, \ r_+ \leq l_+ \\ r_- \geq \nuone l_- \ \text{or} \ r_+ \geq \nuone l_+}}
%	(r_-+1) \alpha^{r_-} \Pp(l_-, l) A(t^{\xi_1},l_-,r_-+1) \
%	(r_+ +1) \alpha^{r_+} \Pp(l_+)A(t^{\xi_1},l_+,r_+ +1)
%	\notag \\
%\lesssim& \sup_{\|y\| \leq t^{\sublin}} \frac{1}{p_t^y} \sum_{l \in J(t-2t^{\xi_1})} q^y_{\iota(y,l)} \Pp^{\bullet}(l)  \notag \\
%	& \prod_{s \in \{-,+\}} \sum_{l_s \in J(t^{\xi_1})} \Pp^s(l_s) \sum_{r_s \in R(s)} (r_s +1)\alpha^{r_s} A(t^{\xi_1},l_s,r_s +1).   \notag
%\end{align}
%
%
%

%===

%======================================
%======================================
\subsection{Proof of Lemma~\ref{lm:main_claim1-4},~\eqref{eq:conv_2} and~\eqref{eq:conv_3}: Convergence for one huge gap at the start or the end}
\label{ssec:gap_end}
%======================================
%======================================

We only show the convergence statement in~\eqref{eq:conv_3} as the proof of~\eqref{eq:conv_2} is analogous. Write 
$$
F_3^{y,t} - T_{0,0}^t + 1 = f_{3;1}^t  + f_{3;2}^t,  
$$
where for $i= 1,2$,
%===
\begin{align*}  
	f_{3;i}^t \coleq &
	\sum_{r \in R^i} \sum_{\substack{\ibf \in H^i(t,r), \zbf}} 
	q^y_{t, \widehat{r+1}}(\ibf,\zbf)
	\prod_{j=1}^r h(z_j, i_j) 
%	f_{3;1}^t \coleq &
%	\sum_{r \in R^1} \sum_{\substack{\ibf \in H_{r,m}^1, \zbf}} 
%	q_r(\ibf,\zbf)
%	\prod_{j=1}^r h(z_j; s_{i_j}, s_{i_j +1}) \\
%	f_{3;2}^t \coleq &
%	\sum_{r \in R^2} \sum_{\substack{\ibf \in H_{r,m}^2, \zbf}} q_r(\ibf,\zbf)\prod_{j=1}^r h(z_j; s_{i_j}, s_{i_j +1}),
\end{align*}
%=== 
and 
$R^1 \coleq \{r \in \N: \ 1 \leq r \leq t^{\xi_1}+1\}$, 
$R^2 \coleq \{r \in \N: \ t^{\xi_1}+1 < r \leq k(t) \}$, \\
$H^1(t,r) \coleq 
		\left\{ \ibf = (i_1, \ldots, i_r)  ~:~
			\substack{ \displaystyle 
				0 \leq i_1 < \ldots < i_{r} \leq r t^{\xi} 
			\\	\displaystyle
				i_r > t^{\xi_2} 
			}
			\right\},
$
$
H_r^2 \coleq I_1(t,r,r+1).
$\\
For $i = 1,2$, one has 
\begin{equation*}     
\left \langle \left( f_{3;i}^t \right)^2 \right \rangle = \sum_{r \in R^i} \lambda^r \sum_{\ibf \in H^i(t,r), \zbf} q^y_{t, \widehat{r+1}}(\ibf, \zbf)^2. 
\end{equation*} 
Convergence in the cases $i=1$ and $i=2$ works then as in the proof of~\eqref{eq:conv_4}. 
%======================================
%======================================

%======================================
%======================================
\subsection{Proof of Lemma~\ref{lm:main_claim5}: Convergence to limiting partition functions}
\label{ssec:convergence_ergodicity}
%======================================
%======================================

Let us first show that the truncated partition function $T_{0,0}^t$ converges to the limiting partition function $Z_{0,0}^\infty$ in the $L^2$ sense and obtain a rate of convergence. We will prove that there is $\theta > 0$ such that 
\begin{equation}    \label{eq:convergence_int_1} 
\lim_{t \to \infty} t^{\theta} \left \langle \left( T_{0,0}^t - Z_{0,0}^t \right)^2 \right \rangle = 0. 
\end{equation} 
One has 
$$
Z_{0,0}^t - T_{0,0}^t = N_1^t + N_2^t,
$$
where
%===
\begin{align}     \label{eq:M_T_diff_1} 
N_1^t
\coleq& \sum_{1 \leq r \leq t^{\xi_1} + 1} \sum_{\substack{\ibf \in I(t,r), \zbf \\ i_r > t^{\xi_2}}} q^y_{t, \widehat{r+1}}(\ibf,\zbf) \prod_{j=1}^r h(z_j, i_j),  \notag \\
N_2^t
\coleq & \sum_{t^{\xi_1} + 1 < r \leq t+1} \sum_{\ibf \in I(t,r), \zbf} q^y_{t, \widehat{r+1}}(\ibf, \zbf) \prod_{j=1}^r h(z_j, i_j).  \notag 
\end{align}
%===
It is then enough to show existence of $\theta > 0$ such that 
\begin{equation}   \label{eq:N_i_convergence} 
\lim_{t \to \infty} t^{\theta} \left \langle \left( N_i^t \right)^2 \right \rangle = 0, \quad i \in \{1,2\}. 
\end{equation} 
We have 
$$
\left \langle \left( N_2^t \right)^2 \right \rangle = \sum_{t^{\xi_1} + 1 < r \leq t+1} \lambda^r \sum_{\ibf \in I(t,r), \zbf} q^y_{t, \widehat{r+1}}(\ibf, \zbf)^2 
\lesssim \sum_{r > t^{\xi_1} + 1} (\alpha_d \lambda)^r, 
$$
so~\eqref{eq:N_i_convergence} holds for $i=2$ and for every $\theta > 0$. Moreover,  
%===
\begin{align*}
\left \langle \left( N_1^t \right)^2 \right \rangle =& \sum_{1 \leq r \leq t^{\xi_1} +1} \lambda^r \sum_{\substack{\ibf \in I(t,r), \xbf \\ i_r > t^{\xi_2}}} q^y_{t, \widehat{r+1}}(\ibf,\zbf)^2, \\
\lesssim& \sum_{1 \leq r \leq t^{\xi_1} +1} \lambda^r \sum_{\substack{t_1, \ldots, t_r \in \N, \xbf \\ t_1 + \ldots + t_r > t^{\xi_2}}} \left(q_{t_1}^{x_1} \right)^2 \ldots \left(q_{t_r}^{x_r} \right)^2, 
\end{align*}
so~\eqref{eq:N_i_convergence} holds for $i=1$ and $\theta \in (0, (\xi_2 - \xi_1) (\tfrac{d}{2}-1))$. 
This implies~\eqref{eq:convergence_int_1}. Combining~\eqref{eq:convergence_int_1} with Theorem~\ref{thm:limiting_part_fun}, one obtains in particular that there is $\theta > 0$ such that 
\begin{equation}    \label{eq:truncated_limiting} 
\lim_{t \to \infty} t^{\theta} \left \langle \left( T_{0,0}^t - Z_{0,0}^{\infty} \right)^2 \right \rangle = 0.
\end{equation} 
To complete the proof of Lemma~\ref{lm:main_claim5}, notice that 
$$
\left \langle \left \lvert T_{0,0}^t T^{y,t}_0 - Z_{0,0}^{\infty} Z_{-\infty}^{y,t} \right \rvert \right \rangle \leq \left \langle \left \lvert T^{y,t}_0 \left(T_{0,0}^t - Z_{0,0}^{\infty} \right) \right \rvert \right \rangle + \left \langle \left \lvert Z_{0,0}^{\infty} \left(T^{y,t}_0 - Z_{-\infty}^{y,t} \right) \right \rvert \right \rangle. 
$$
Therefore, we obtain the desired result by applying Cauchy--Schwarz to the two summands on the right, and using~\eqref{eq:truncated_limiting} together with  
$$
\lim_{t \to \infty} \left \langle \left( T_{0,0}^t \right)^2 \right \rangle = \left \langle \left( Z_{0,0}^{\infty} \right)^2 \right \rangle < \infty.
$$

%======================================
%======================================

%=====================================
%=====================================

%======================================
%======================================
\begin{appendix}
%======================================
%======================================

%======================================
%======================================
\section{Proofs of estimates for transition probabilities}
\label{ssec:transition_prob_proofs}
%======================================
%======================================

%======================================
%======================================
\subsection{Proof of Lemma~\ref{lm:Law_Lim}}%======================================
%======================================

For $t \in \N_0$, set $\gamma_t^* \coleq \gamma_{2t}$.
Then $\gamma^*$ is a random walk on the lattice $(\Z^d)_{\text{ev}}$ consisting of those points in $\Z^d$ whose coordinate sum is even.
If $\{ e_j \}_{1 \leq j \leq d}$ is the standard basis for $\mathbb{R}^d$, then $\{e_1 + e_j: 1 \leq j \leq d\}$ is a basis for $(\Z^d)_{\text{ev}}$.
Let $L : \mathbb{R}^d \to \mathbb{R}^d$ be the linear transformation mapping $e_1 + e_j$ to $e_j$ for $1 \leq j \leq d$, and define $\tilde{\gamma}_t \coleq L \gamma_t^\ast$.
Then, $\tilde \gamma$ is an aperiodic, irreducible, symmetric random walk on $\Z^d$ with bounded increments, so it satisfies the conditions of Theorem~2.3.11 in~\cite{Lawler_Limic}.
Thus, there is $\rho > 0$ such that for any $i \in \N$ and for any $z \in \Z^d$ satisfying $\|z\| < \rho i$, we have 
$$
	\tilde q_i^z 
		\coleq \Pp(\tilde \gamma_i = z) 
		= 2 \left( \frac{d}{4 \pi i} \right)^{\frac{d}{2}} \exp \left( - \frac{d}{4i} \|L^{-1} z \|^2 \right) \exp \left(O \left(\frac{1}{i} + \frac{\|z\|^4}{i^3} \right) \right). 
$$
Now, we fix $\sublin \in (\tfrac{3}{4}, 1)$, $\tilde{\sublin} \in (\sublin,1)$, and let $T \in \N$ be so large that $1 + t^{\sublin} < (t-1)^{\tilde \sublin}$ and $t^{\tilde \sublin} < \tfrac{\rho}{2  \vertiii{L} } t$ for all $t \geq T$, where $\vertiii{L}$ is the operator norm of $L$. We distinguish between two cases: $t$ is either even or odd.\\

\noindent \textsc{Even case.} If $t=2m$ for some $m \in \N$ then we can prove a slightly stronger statement:

\begin{claim}\label{claim:even_bound}
There are constants $c_1, c_2 > 0$, independent of $\sublin$ and $\tilde{\sublin}$, such that \eqref{eq:lclt_lower_bound} holds for every even $t \geq T$ and $y \in \Z^d$ with $q^y_{t} > 0$ and $\|y\| \leq {t}^{\tilde \sublin}$.
\end{claim}

The difference to the conclusion of Lemma~\ref{lm:Law_Lim} is that the estimate holds for $\| y \| \leq t^{\tilde \sublin}$ and not just for $\|y\| \leq t^{\sublin}$. To prove this claim, fix $t=2m \geq T$, $y \in \Z^d$ such that $q^y_{2m} > 0$ and $\|y\| \leq (2m)^{\tilde \sublin}$.
Then 
$q^y_{2m} = \tilde q^{Ly}_m.$
Since 
$
	\|Ly \| \leq \vertiii{L} \|y\| \leq \vertiii{L} t^{\tilde \sublin} < \rho m, 
$
one has 
\begin{flalign*}
	\tilde{q}^{Ly}_m	 &= 2 \left(\tfrac{d}{2 \pi t} \right)^{\frac{d}{2}} 
				\exp \big( -\tfrac{d}{2t} \|y\|^2 \big)
			\exp \left(O \left(\frac{1}{m} + \frac{\|Ly\|^4}{m^3} \right) \right)
\\		&\geq c_1 \left(\tfrac{d}{2 \pi t} \right)^{\frac{d}{2}} 
				\exp \big( -\tfrac{d}{2t} \|y\|^2 \big)  \exp \big(-c_2 t^{4 \tilde \sublin -3} \big)
\end{flalign*}
for some constants $c_1, c_2 > 0$.\\

\noindent \textsc{Odd case.}
Now, suppose $t = 2m+1 \geq T$ for some $m \in \N$.
Fix $y \in \Z^d$ such that $q^y_{2m+1} > 0$ and $\|y\| \leq (2m+1)^{\sublin}$. Let $E$ be the set of standard unit vectors in $\R^d$ and their additive inverses.
Then
$$
q^y_{2m+1} = \sum_{z \in \Z^d} q^{y-z}_{2m} q^z_1 = \frac{1}{2d} \sum_{z \in E} q^{y-z}_{2m}. 
$$
Since $\|y-z\| < 1 + t^{\sublin} < (t-1)^{\tilde \sublin} = (2m)^{\tilde \sublin}$ and $q^{y-z}_{2m} > 0$ for all $z \in E$, then using Claim~\ref{claim:even_bound}, we can bound $q^y_{2m+1}$ from below as follows: There are $c'_1, c'_2 > 0$ such that
\begin{flalign}
	q^{y}_{2m+1}
		&\geq \frac{c'_1}{2d} \left(\tfrac{d}{4 \pi m} \right)^{\frac{d}{2}} 
		\exp \Big(-c'_2 (2m)^{4 \tilde \sublin -3} \Big) 
		\sum_{z \in E} 
			\exp \big( -\tfrac{d}{4m} \|y - z\|^2 \big) 
\notag
\\		&\geq \frac{c'_1}{2d} \left(\tfrac{d}{2 \pi t} \right)^{\frac{d}{2}}
		\exp \Big(-c'_2 t^{4 \tilde \sublin -3} \Big) 
		\exp \big( -\tfrac{d}{4m} \|y - e_1\|^2 \big).
\label{eq:sum_phi}
\end{flalign}
In addition to $t \geq T$, assume that $t$ is so large that 
$$
\exp \left(-\frac{d}{2} \left(\frac{t^{2 \sublin}}{t(t-1)} + \frac{1+ 2 t^{\sublin}}{t-1} \right) \right)  > \frac{1}{2}. 
$$
Since 
$
\|y - e_1 \|^2 = \|y\|^2 + 1 - 2 y \cdot e_1 \leq \|y\|^2 + 1 + 2 \|y\|, 
$
it follows that
\begin{align*}
\exp \big( -\tfrac{d}{4m} \|y - e_1\|^2 \big)
	\geq& \exp \Big(- \tfrac{d}{4m} \big(\|y\|^2 + 1 + 2 \|y\| \big) \Big) \\
	\geq& \exp \Big( - \tfrac{d}{2t} \| y \|^2 \Big) \exp \Big(-\tfrac{d}{2} \Big(\tfrac{t^{2 \sublin}}{t(t-1)} + \tfrac{1 + 2 t^{\sublin}}{t-1} \Big) \Big) \\
	>& \tfrac{1}{2} \exp \Big( - \tfrac{d}{2t} \| y \|^2 \Big).
\end{align*}
Plugging this into~\eqref{eq:sum_phi}, we obtain the desired estimate.

%======================================
%======================================

%======================================
%======================================
\subsection{Proof of Lemma~\ref{lm:lin_func_upper_bd}} 
%======================================
%======================================
Recall from Section~\ref{sec:transition_prob} that $\theta^0 = (0,\ldots,0)$ and $\theta^1 = (\pi,\ldots,\pi)$.
For any $\varepsilon > 0$ and $j\in\{0,1\}$, let $\mathcal{D}_j^\varepsilon \coleq \{ \theta \in \R^d ~:~ \| \theta - \theta^j \| < \varepsilon \}$. Let $\varphi$ be a linear functional on $\R^d$ such that $\lvert \varphi(x) \rvert \leq \| x \|, x \in \R^d$, and let $\Phi$ be the corresponding function defined in~\eqref{eq:phi_def}. 

\begin{claim}\label{claim:exp_phi_bound}
There exist $\varepsilon, \delta > 0$ such that, for $j \in \{0,1\}$,
\begin{equation*}
\left| \frac{ \Phi (\theta) }{\Phi ( \theta^j ) } \right|
	\leq e^{- \delta \| \theta - \theta^j \|^2}
\qquad\:\quad
\rlap{for all $\theta \in \Cc \setminus \mathcal{D}_{1-j}^\varepsilon$.}
\end{equation*}
%%
%\begin{equation}       \label{eq:exp_phi_bound}
%\begin{aligned}
%\left| \frac{ \Phi (\theta) }{\Phi ( 0 ) } \right|
%	&\leq e^{- \delta \| \theta \|^2}
%\qquad\:\quad
%\rlap{for all $\theta \in \Cc \setminus \mathcal{D}_1^\varepsilon$,}
%\\
%\left| \frac{ \Phi (\theta) }{\Phi (\theta^1) } \right|
%	&\leq e^{- \delta \| \theta - \theta^1 \|^2}
%\qquad
%\rlap{for all $\theta \in \Cc \setminus \mathcal{D}_0^\varepsilon$.}
%\end{aligned}
%\end{equation}
%%
\end{claim}

\paragraph*{\textit{Proof of Claim~\ref{claim:exp_phi_bound}.}} 
For each $j \in \{0,1\}$, define scaled versions of the gradient vector and the Hessian matrix of $\Phi$ at $\theta^j$:
$$
	G_j
		\coleq -i \frac{ \nabla \Phi (\theta^j) }{\Phi (\theta^j)}
\qqtext{and}
	H_j
		\coleq - \frac{1}{2} \frac{\nabla^2 \Phi (\theta^j)}{\Phi (\theta^j)}.
$$
A simple computation shows that the matrix $H_j$ is diagonal, and that for every $l \in \{ 1, \ldots, d \}$, the $l$-th component of $G_j$ and the $(l, l)$-entry of $H_j$ are, respectively,
\begin{equation}\label{eq:G_and_H_defns}
	G_j^l
		= \frac{\sinh (\varphi_l)}{d \Phi (0)}
\qqtext{and}
	H_j^l
		= \frac{\cosh (\varphi_l)}{2d \Phi (0)}.
\end{equation}
If we Taylor expand $\Phi$ around $\theta^j$, we obtain 

\begin{flalign*}
	\left| \frac{\Phi (\theta)}{\Phi (\theta^j)} \right|
		&= \Big| 1 + i G_j \cdot (\theta - \theta^j) 
				- (\theta - \theta^j) \cdot H_j (\theta - \theta^j) + O \big( \| \theta - \theta^j \|^3 \big) \Big|
\\		&= \Big( 1 - 2 (\theta - \theta^j) \cdot H_j (\theta - \theta^j) + (G_j \cdot (\theta - \theta^j))^2 + O \big( \| \theta - \theta^j \|^3 \big) \Big)^{1/2}.
\end{flalign*}

Here and in the sequel, $g(\theta) = O(f(\theta))$ means there is a constant $c > 0$, independent of $\varphi$, such that $\lvert g(\theta) \rvert \leq c f(\theta)$.
In the Taylor expansion above, the constant $c$ corresponding to the error term $O\big( \| \theta - \theta^j \|^3 \big)$ may be chosen independently of $\varphi$ because of the assumption that $\|\varphi\| \leq 1$.
Notice from~\eqref{eq:G_and_H_defns} that $G_0 = G_1$ and $H_0 = H_1$, so in order to prove Claim~\ref{claim:exp_phi_bound}, it is enough to consider the case $j = 0$, where $\theta^j = (0, \ldots, 0)$.
If we write $\theta = (\theta_1, \ldots, \theta_d)$, then, using Jensen's inequality for sums,
$$
	\big(G_0 \cdot \theta \big)^2
		\leq \frac{1}{d \Phi (0)} \sum_{l=1}^d \sinh(|\varphi_l|) \ \theta_l^2.
$$
Using the expression for $H_0$ in \eqref{eq:G_and_H_defns} as well as $\| \varphi \| \leq 1$, we obtain 
$$
	2 \theta \cdot H_0 \theta - \big( G_0 \cdot \theta \big)^2
	\geq \frac{1}{d \Phi (0)} \sum_{l=1}^d e^{- |\varphi_l|} \theta_l^2
	\geq \frac{1}{de \Phi (0)}  \| \theta\|^2.
$$
Thus, there are $\varepsilon > 0$ and a constant $c > 0$ such that for all $\| \theta \| \leq \varepsilon$,
$$
	\left| \frac{\Phi (\theta)}{\Phi (0)} \right|
		\leq \left( 1 - \frac{1}{de \Phi (0)}  \| \theta\|^2 + O \big( \| \theta \|^3 \big) \right)^{1/2}
		\leq \big( 1 - c \| \theta \|^2 \big)^{1/2}.
$$
Since the map $\theta \mapsto \big| \Phi (\theta) / \Phi (0) \big|$ is continuous and strictly less than $1$ for all $\theta \in \Cc$ except $\theta^0, \theta^1$, it follows that 
$$
\mathscr{s}(\varphi) \coleq \sup \left\{\big| \Phi (\theta) / \Phi (0) \big| : \ \theta \in \Cc; \ \| \theta\|, \| \theta - \theta^1\| \geq \varepsilon \right\} < 1.
$$
In fact, one even has $\sup_{\| \varphi \| \leq 1} \mathscr{s}(\varphi) < 1$. Hence, if we choose $\tilde{c} \in (0, c)$ so small that $\big( 1 - \tilde{c} \| \theta \|^2 \big) \geq (\sup_{\|\varphi \| \leq 1} \mathscr{s}(\varphi))^2$ for all $\theta \in \Cc$, then Claim~\ref{claim:exp_phi_bound} follows with $\delta \coleq \tilde{c} / 2$.

\medskip

For $t \in \N$, let $\widehat{\Phi^t}$ be the Fourier transform of $\Phi^t$; i.e.,
$$ 
\widehat{\Phi^t}(z) \coleq \frac{1}{(2 \pi)^d} \int_{\Cc} \Phi(\theta)^t e^{-i \theta \cdot z} \, d \theta, \quad z \in \Z^d. 
$$
Since $\Phi (\theta)^t = \E \left[ e^{i \theta \cdot \gamma_t} e^{\varphi(\gamma_t)} \right],$ one has
\begin{equation}\label{eq:Fourier} 
	\widehat{\Phi^t}(z) 
		= \sum_{y \in \Z^d} \Pp(\gamma_t = y)
			e^{\varphi(y)} \frac{1}{(2 \pi)^d} 
			\int_{\Cc} e^{i  \theta \cdot (y-z)} \, d \theta
		= q_t^z e^{\varphi(z)}.
\end{equation}

Now, we estimate with the help of~\eqref{190918172610} and Claim~\ref{claim:exp_phi_bound}: 
\begin{flalign*}
	q_t^z e^{\varphi(z)}
		&\leq \frac{1}{(2\pi)^d} \int_\Cc \big| \Phi (\theta) \big|^t \, d\theta
		= \frac{1}{(2\pi)^d} \int_\Cc \left| \frac{\Phi (\theta)}{\Phi (0)} \right|^t \, d\theta
			\ \sum_{y \in \Z^d} q_t^y e^{\varphi (y)}
\\		&\leq \frac{1}{(2 \pi)^d}
		\left( \int_{\Cc \setminus \mathcal{D}_1^\varepsilon} e^{- \delta t \| \theta \|^2} \, d\theta
			+ \int_{\Cc \setminus \mathcal{D}_0^\varepsilon} e^{- \delta t \| \theta - \theta^1 \|^2} \, d\theta \right)
			\sum_{y \in \Z^d} q_t^y e^{\varphi (y)}
\\
		&\leq \frac{2}{(2\pi)^d} \int_{\R^d}
				e^{- \delta t \| \theta \|^2} \, d\theta
			\ \sum_{y \in \Z^d} q_t^y e^{\varphi (y)}.
\end{flalign*}
Finally, for some constant $C > 0$, 
$$
\int_{\R^d} e^{-\delta t \| \theta \|^2} \, d\theta \leq C \int_0^{\infty} r^{d-1} e^{-\delta t r^2} \, dr 
	= C t^{-\frac{d}{2}} \int_0^{\infty} \rho^{d-1} e^{-\delta \rho^2} \, d \rho. 
$$

%======================================
%======================================
\subsection{Proof of Lemma~\ref{lm:linear_functional}} 
%======================================
%======================================

For $j \in \{0,1\}$, let $\mathcal{B}_j \coleq \{ \theta \in \Cc ~:~ \| \theta - \theta^j \| \leq t^{-2/5} \}$.
Recall from~\eqref{eq:Fourier} that for all $z \in \Z^d$, $t \in \N$, and for every linear functional $\varphi$ on $\R^d$ satisfying $\| \varphi \| \leq 1$, one has 
\begin{equation*}  
	q_t^z e^{\varphi (z)} 
		= \frac{1}{(2 \pi)^d} \int_{\Cc} \Phi(\theta)^t e^{-i \theta \cdot z} \, d \theta
		= I_0 + I_1 + I,
\end{equation*}
where
$$
	I_j \coleq \frac{1}{(2 \pi)^d} \int_{\mathcal{B}_j} \Phi(\theta)^t e^{-i \theta \cdot z} \, d \theta
\qtext{and}
	I \coleq \frac{1}{(2 \pi)^d} \int_{\Cc \setminus (\mathcal{B}_0 \cup \mathcal{B}_1) } \Phi(\theta)^t e^{-i \theta \cdot z} \, d \theta.
$$
Then we find
\begin{align}     \label{eq:est_sum_squares_4}
\frac{1}{(2 \pi)^d} \int_{\Cc} 
	\big| \Phi (\theta) \big|^t \, d \theta 
	- q^z_t e^{\varphi (z)}
\leq& \sum_{j \in \{0, 1\}}
	\left( \frac{1}{(2 \pi)^d} 
		\int_{\mathcal{B}_j} \big| \Phi (\theta) \big|^t \, d \theta - \Real (I_j) \right) 
\\ 
& + \frac{2}{(2\pi)^d} \int_{\Cc \setminus (\mathcal{B}_0 \cup \mathcal{B}_1)}
			\big| \Phi (\theta) \big|^t \, d \theta.   \notag 
\end{align}
We now estimate the expression on the right-hand side. First, we show that the linear functional $\varphi$ can be chosen in such a way that for $j \in \{0,1\}$, 
\begin{equation}     \label{eq:phi_minus_real} 
\frac{1}{(2 \pi)^d} \int_{\Bc_j} \big| \Phi(\theta) \big|^t \, d\theta - \Real(I_j) = O(t^{-2/5}) \int_{\Cc} \big| \Phi(\theta) \big|^t \, d\theta. 
\end{equation} 
The idea is to choose $\varphi$ as a function of $z$ and $t$ in such a manner that the linear term in the Taylor expansion of $\Phi(\theta) e^{-\frac{i}{t} z \cdot \theta}$ around $\theta^{j}$ vanishes, i.e. 
$$
\Phi(\theta^j) \nabla \left(e^{-\frac{i}{t} z \cdot \theta^j} \right) + e^{-\frac{i}{t} z \cdot \theta^j} \nabla \Phi(\theta^j) = 0. 
$$
If we denote the $k$th component of $z$ by $z_k$, this is equivalent to 
\begin{equation*}
\frac{z_k}{t} = \frac{\sinh(\varphi_k)}{d \Phi(0)} = \frac{\sinh(\varphi_k)}{\sum_{l=1}^d \cosh(\varphi_l)}, \quad 1 \leq k \leq d,
\end{equation*}
by virtue of~\eqref{eq:G_and_H_defns}. 
Let $F:\R^d \to \R^d$ be given by
$$
F(x_1, \ldots, x_d) \coleq \sum_{k=1}^d  \frac{\sinh(x_k)}{\sum_{l=1}^d \cosh(x_l)}e_k.
$$ 

For $r > 0$ and $x \in \R^d$, let $B_r(x)$ denote the open Euclidean ball of radius $r$ centered at $x$. Since $F(0) = 0$ and 
$$
\det DF(0) = \frac{1}{d^d} \neq 0, 
$$
the inverse function theorem yields existence of $\rho_1 > 0$ and an open neighborhood $U$ of $0$ such that $F: U \to B_{\rho_1}(0)$ is a diffeomorphism. 
Therefore, for any $t \in \N$ and $z \in \Z^d$ with $\|z\| < \rho_1 t$, there is $\varphi \in U$ such that $F(\varphi) = \tfrac{z}{t}$. Since $F^{-1}$ is differentiable and $F^{-1}
(0) = 0$, there is $\rho_2 >0$ such that 
\begin{equation*}
\| \varphi \| = \| F^{-1} (\tfrac{z}{t}) \| \leq \rho_2 \frac{\|z\|}{t}. 
\end{equation*}
Without loss of generality, we may assume that $\rho_1 \rho_2 \leq 1$ so that $\|\varphi\| \leq 1$. 

Fix $t \in \N$, $z \in \Z^d$ such that $\|z\| \leq \rho_1 t$ and $q^z_t > 0$, and the corresponding $\varphi \in \R^d$ such that $F(\varphi) = \tfrac{z}{t}$. We identify $\varphi$ with the linear functional mapping $e_k$ to $\varphi_k$ for $1 \leq k \leq d$. For this choice of $\varphi$, the linear term in the Taylor expansion of $\Phi(\theta) e^{-\frac{i}{t} z \cdot \theta}$ vanishes, so we have for $j \in \{0,1\}$ and $\theta \in \mathcal{B}_j$ ($\|\theta - \theta^j \| \leq t^{-\frac{2}{5}}$)
\begin{flalign}
	\Phi (\theta) e^{- \frac{i}{t} z \cdot \theta}
		&= \Phi (\theta^j) e^{- \frac{i}{t} z \cdot \theta^j}
		+ (\theta - \theta^j) \cdot A_j (\theta - \theta^j) + O \big( \| \theta - \theta^j \|^3 \big)
\notag
\\		&= \Phi (\theta^j) e^{- \frac{i}{t}  z \cdot \theta^j} 
			\left( 1 + \frac{(\theta - \theta^j) \cdot  A_j (\theta - \theta^j)}{\Phi (\theta^j) e^{- \frac{i}{t} z \cdot \theta^j}} + O (t^{-6/5}) \right),
			\label{190917112559}
\end{flalign}
where $A_j$ is the quadratic form in the Taylor expansion of $\Phi (\theta) e^{- \frac{i}{t} z \cdot \theta}$.
The error term $O(t^{-6/5})$ is complex-valued, whereas the entries of $A_j \big/ \Phi (\theta^j) e^{- \frac{i}{t} z \cdot \theta^j}$ are real.
Let $x_j(\theta)$ and $y_j(\theta)$ denote respectively the real and imaginary part of 
$$
1 + \frac{(\theta - \theta^j) \cdot A_j (\theta - \theta^j)}{\Phi(\theta^j) e^{-\frac{i}{t} z \cdot \theta^j}} + O(t^{-6/5}). 
$$
Then the left-hand side of~\eqref{eq:phi_minus_real} can be written as follows:
\begin{equation}
\label{190917113611}
	\frac{\Phi (0)^t}{(2\pi)^d}
		\int_{\mathcal{B}_j} 
			\bigg( \big| x_j (\theta) + i y_j (\theta) \big|^t
					- \Real \Big( \big( x_j (\theta) + i y_j (\theta) \big)^t \Big) \bigg) \, d\theta.
\end{equation}
Here, in the case $j=1$, we used the assumption that $q^z_t > 0$ and hence $t$ and $\|z\|_1$ have the same parity: as $t \equiv \| z \|_1$, one has $\Phi (\theta^1)^t e^{-i z \cdot \theta^1} = \Phi (0)^t (-1)^t e^{- i \pi \| z \|_1} = \Phi (0)^t$.
If we represent $x_j(\theta) + i y_j(\theta)$ in polar form, then the modulus is $\big| \Phi (\theta) / \Phi (0) \big|$ and the argument is of order $O (t^{-6/5})$.
As a result, the integrand in~\eqref{190917113611} can be written as 
$$
	\frac{ \big| \Phi (\theta) \big|^t}{\Phi (0)^t}
		\Big(1 - \cos \big(O(t^{-1/5}) \big) \Big)
	= \frac{ \big| \Phi (\theta) \big|^t}{\Phi (0)^t}
		O (t^{-2/5}), 
$$
which yields~\eqref{eq:phi_minus_real}. 

\medskip 

We continue estimating the expression in~\eqref{eq:est_sum_squares_4} by showing that for $F(\varphi) = \tfrac{z}{t}$, one also has 
\begin{equation}    \label{eq:C_minus_B_int_est} 
\frac{2}{(2 \pi)^d} \int_{\Cc \setminus (\mathcal{B}_0 \cup \mathcal{B}_1)} \big| \Phi(\theta) \big|^t \, d\theta \lesssim t^{-2/5} \int_{\Cc} \big| \Phi(\theta) \big|^t \, d\theta. 
\end{equation}  
By Claim~\ref{claim:exp_phi_bound}, there are $\varepsilon, \delta > 0$ such that the left-hand side of~\eqref{eq:C_minus_B_int_est} is dominated by
\begin{equation*} 
\frac{2 \Phi (0)^t}{(2\pi)^d}
	\left( \int_{\Cc \setminus (\mathcal{B}_0 \cup \mathcal{D}_1^\varepsilon)} 
	 \hspace{-3mm} e^{- \delta t \| \theta \|^2} d\theta
	+ \int_{\Cc \setminus (\mathcal{B}_1 \cup \mathcal{D}_0^\varepsilon)} 
	 \hspace{-3mm} e^{- \delta t \| \theta - \theta^1 \|^2} d\theta
	 \right)
\lesssim \Phi (0)^t e^{- \delta t^{1/5}} \hspace{-1mm}. 
\end{equation*}
Here we used that 
$$
\Cc \setminus (\mathcal{B}_0 \cup \mathcal{B}_1) \subseteq \Cc \setminus \big[ (\mathcal{B}_0 \cap \mathcal{D}^{\varepsilon}_0) \cup (\mathcal{B}_1 \cap \mathcal{D}^{\varepsilon}_1) \big] \subseteq \big[ \mathcal{C} \setminus (\mathcal{B}_0 \cup \mathcal{D}^{\varepsilon}_1) \big] \cup \big[ \mathcal{C} \setminus (\mathcal{B}_1 \cup \mathcal{D}^{\varepsilon}_0) \big].
$$
As $e^{-\delta t^{1/5}} \lesssim t^{-2/5} t^{-d/2}$, the estimate in~\eqref{eq:C_minus_B_int_est} follows once we show that 
\begin{equation}    \label{eq:J_lower_bound} 
J_t \coleq \int_{\Cc} \left( \frac{\lvert \Phi(\theta) \rvert}{\Phi(0)} \right)^t \, d\theta \gtrsim t^{-d/2}. 
\end{equation}   
We have   
\begin{flalign*}
		J_t \geq& \frac{1}{\Phi (0)^t}
				\int_{\mathcal{B}_0} \big| \Phi (\theta) e^{-\frac{i}{t} z \cdot \theta} \Big|^t \, d \theta
			= \int_{\mathcal{B}_0} \big| x_0 (\theta) + i y_0 (\theta) \big|^t \, d\theta
\\		& \geq \cos \big( O (t^{-1/5}) \big) \int_{\mathcal{B}_0} \big| x_0 (\theta) \big|^t \, d \theta
		= \big( 1 + O (t^{-2/5}) \big)
			\int_{\mathcal{B}_0} \big| x_0 (\theta) \big|^t \, d \theta, 
\intertext{where we used~\eqref{190917112559}. For $\theta \in \mathcal{B}_0$, one has $x_0 (\theta) = \exp \big((\theta \cdot A_0 \theta) / \Phi (0) \big) \big(1 + O (t^{-6/5}) \big)$, so we can continue the above chain of inequalities as follows:}
		&= \big( 1 + O (t^{-2/5}) \big)
					\big( 1 + O (t^{-6/5}) \big)^t
				\int_{\mathcal{B}_0} \exp \left(t \frac{\theta \cdot A_0 \theta}{\Phi (0)} \right) \, d \theta
\\		& \gtrsim \int_{\mathcal{B}_0} \exp \left(t \frac{\theta \cdot A_0 \theta}{\Phi (0)} \right) \, d \theta \gtrsim \int_{\mathcal{B}_0} e^{-c t \| \theta \|^2 } \, d \theta \gtrsim t^{-d/2}, 
\end{flalign*}
where $c > 0$ is some constant.
Combining~\eqref{eq:est_sum_squares_4},~\eqref{eq:phi_minus_real}, and~\eqref{eq:C_minus_B_int_est} yields 
\begin{equation}    \label{eq:q_e_lower_bound} 
q_t^z e^{\varphi(z)} \geq \left(1 + O(t^{-2/5}) \right) \frac{1}{(2 \pi)^d} \int_{\Cc} \big| \Phi(\theta) \big|^t \, d\theta
\end{equation} 
and hence 
$$
\frac{1}{(2 \pi)^d} \int_{\Cc} \big| \Phi(\theta) \big|^t \, d\theta \leq \left(1 + O(t^{-2/5}) \right) q^z_t e^{\varphi(z)}. 
$$

\bigskip

To show that 
$$
q^z_t e^{\varphi(z)} \gtrsim t^{-\frac{d}{2}} \sum_{y \in \Z^d} q^y_t e^{\varphi(y)},  
$$
one simply combines~\eqref{eq:q_e_lower_bound} with~\eqref{eq:J_lower_bound} and~\eqref{190918172610}. 

%======================================
%======================================
\subsection{Proof of Lemma~\ref{lm:ratio_estimate}}
%======================================
%======================================
Let $\rho \coleq \rho_1$ and $\rho_2$ be as in Lemma~\ref{lm:linear_functional}, and let $t, t' \in \N$, $z, z' \in \Z^d$ such that $\|z\| \leq \rho t$ and $q^z_t > 0$. Let $\varphi$ be the linear functional from Lemma~\ref{lm:linear_functional} that corresponds to $t$ and $z$, and for which $\| \varphi \| \leq \rho_2 \tfrac{\|z\|}{t}$ and 
\begin{equation}   \label{eq:C_int_est} 
\frac{1}{(2 \pi)^d} \int_{\Cc} \lvert \Phi (\theta) \rvert^t \ d \theta \leq \left(1 + O(t^{-\frac{2}{5}}) \right) q_t^z e^{\varphi(z)}. 
\end{equation}
We consider two cases: $t' > t$ and $t' \leq t$.\\

\noindent \textsc{Case ``$t' > t$''.}
By~\eqref{eq:Fourier} and~\eqref{eq:phi_theta_bd}, one has   
$$ 
q_{t'}^{z'} e^{\varphi(z')} \leq \frac{1}{(2 \pi)^{d}} \int_{\Cc} \lvert \Phi (\theta) \rvert^{t'} \ d \theta \leq \Phi (0)^{t'-t} \frac{1}{(2 \pi)^d} \int_{\Cc} \lvert \Phi (\theta) \rvert^t \ d \theta. 
$$ 
Furthermore,  
$$ 
\Phi (0)^{t'-t} \leq e^{\|\varphi\| (t'-t)} \leq e^{\rho_2 \frac{\|z\|}{t} (t'-t)}. 
$$ 
The estimate in~\eqref{eq:C_int_est} then implies   
\begin{align*} 
\frac{q_{t'}^{z'}}{q_t^z} \leq& \left(1 + O(t^{-\frac{2}{5}}) \right) e^{\varphi(z-z')} e^{\rho_2 \frac{\|z\|}{t} (t'-t)}  \\
\leq& \left(1+O(t^{-\frac{2}{5}}) \right) \exp \left(\rho_2 \frac{\|z\|}{t} \left( \|z-z'\| + \lvert t' - t \rvert \right) \right). 
\end{align*}

\smallskip 

\noindent \textsc{Case ``$t' \leq t$''.}
If $t' \leq t$, the function $x \mapsto x^{t/ t'}$ is convex, and Jensen's inequality implies   
\begin{equation}     \label{eq:convex_Jensen} 
q_{t'}^{z'} e^{\varphi(z')}  \leq \biggl( \frac{1}{(2 \pi)^d} \int_{\Cc} \lvert \Phi (\theta) \rvert^t \ d \theta \biggr)^{t'/t} = \Phi (0)^{t'} J_t^{t'/t} \leq \Phi (0)^t J_t^{t'/t}, 
\end{equation}    
where $J_t$ was defined in~\eqref{eq:J_lower_bound}. 
Since $J_t \gtrsim t^{-d/2}$,   
\begin{equation}    \label{eq:est_sum_squares_3}
J_t^{t'/t} \leq \left(c_2 t^{\frac{d}{2}} \right)^{\frac{t-t'}{t}} J_t \leq \exp \left( c_3 \ln(t) \frac{t-t'}{t} \right) J_t
\end{equation}
for some constants $c_2, c_3>0$.  Combining~\eqref{eq:convex_Jensen} and~\eqref{eq:est_sum_squares_3}, we obtain 
\begin{equation*}  
q_{t'}^{z'} \leq e^{-\varphi(z')} \frac{1}{(2 \pi)^d} \int_{\Cc} \lvert \Phi (\theta) \rvert^t  \ d \theta \ \exp \left(c_3 \ln(t) \frac{t-t'}{t} \right).  
\end{equation*}  
Together with~\eqref{eq:C_int_est}, this yields  
\begin{align*}
\frac{q_{t'}^{z'}}{q_t^z} \leq& \left(1 + O(t^{-\frac{2}{5}}) \right) \exp \left(\varphi(z - z')+ c_3 \ln(t) \frac{t - t'}{t} \right) \\
\leq& \left(1 + O(t^{-\frac{2}{5}}) \right) \exp \left(c \left( \frac{\|z\|}{t} \| z - z' \| + \ln(t) \frac{t -t'}{t} \right) \right) 
\end{align*}
for some constant $c > 0$. 

\section{A calculus estimate}
\label{appendix_calc}
%======================================
%======================================

%======================================
%======================================
%\subsection{A Lemma About Numbers}
%======================================
%======================================

\begin{lemma}             \label{lm:lemma_3}
There is $c > 0$ such that for any $t \in \N$, $l \in \N_0$, and $M > 0$,  
\begin{equation}   \label{eq:calc_estim} 
\sum_{\substack{t_1 + \ldots + t_{l+1} = t, \\ t_1, \ldots, t_{l+1} \geq M}} \prod_{j=1}^{l+1} t_j^{-\frac{d}{2}} \leq \frac{c^l}{M^{l(\frac{d}{2}-1)}} t^{-\frac{d}{2}}.
\end{equation}
The sum on the left is taken over all positive integers $t_1, \ldots, t_{l+1}$ that satisfy the two conditions under the summation sign. 
\end{lemma}
%\noindent We prove Lemma~\ref{lm:lemma_3} in Subsection~\ref{appendix_calc}.

\begin{proof}
%[Proof of Lemma~\ref{lm:lemma_3}] 
We choose 
$$
c \coleq 2^d \max\left\{\zeta(\tfrac{d}{2}); (\tfrac{d}{2}-1)^{-1}\right\}, 
$$
where $\zeta$ is the Riemann zeta function, and prove the statement by induction. In the base case $l=0$, the left side of~\eqref{eq:calc_estim} is either zero (if $t < M$), or becomes 
$$ 
t^{-\frac{d}{2}} = \frac{c^0}{M^0} t^{-\frac{d}{2}}.   
$$ 
In the induction step, suppose that~\eqref{eq:calc_estim} holds for some $l \in \N_0$. Then, 
\begin{equation}   \label{eq:lemma_3_1}
\sum_{\substack{t_1 + \ldots + t_{l+2} = t, \\ t_1, \ldots, t_{l+2} \geq M}} \prod_{j=1}^{l+2} t_j^{-\frac{d}{2}} = \sum_{\substack{t' + t_{l+2} = t, \\ t', t_{l+2} \geq M}} \biggl( \sum_{\substack{t_1 + \ldots + t_{l+1} = t', \\ t_1, \ldots, t_{l+1} \geq M}} \prod_{j=1}^{l+1} t_j^{-\frac{d}{2}} \biggr) t_{l+2}^{-\frac{d}{2}}. 
\end{equation}  
For any $t'$, 
$$ 
\sum_{\substack{t_1 + \ldots + t_{l+1} = t', \\ t_1, \ldots, t_{l+1} \geq M}} \prod_{j=1}^{l+1} t_j^{-\frac{d}{2}} \leq  \frac{c^l}{M^{l (\frac{d}{2}-1)}} (t')^{-\frac{d}{2}}
$$ 
by induction hypothesis. Hence, the right side of~\eqref{eq:lemma_3_1} is bounded from above by 
\begin{equation}    \label{eq:lemma_3_3}
\frac{c^l}{M^{l(\frac{d}{2}-1)}} \sum_{\substack{t' + t_{l+2} = t, \\ t', t_{l+2} \geq M}} (t')^{-\frac{d}{2}} t_{l+2}^{-\frac{d}{2}}. 
\end{equation} 
We have  
$$ 
\sum_{\substack{t' + t_{l+2} = t, \\ t', t_{l+2} \geq M}} (t')^{-\frac{d}{2}} t_{l+2}^{-\frac{d}{2}} \leq 2 \sum_{\substack{t' + t_{l+2} = t, \\ t' \geq t_{l+2} \geq M}} (t')^{-\frac{d}{2}} t_{l+2}^{-\frac{d}{2}}. 
$$
If $t' + t_{l+2} = t$ and $t' \geq t_{l+2}$, it follows that $t' \geq \tfrac{t}{2}$, so the expression on the right is bounded from above by  
\begin{equation}   \label{eq:lemma_3_2} 
2^{\frac{d}{2}+1} t^{-\frac{d}{2}} \sum_{t_{l+2} \geq M} t_{l+2}^{-\frac{d}{2}}. 
\end{equation} 
If $M \geq 2$, we have  
$$ 
\sum_{t_{l+2} \geq M} t_{l+2}^{-\frac{d}{2}}  \leq \int_{\frac{M}{2}}^{\infty} x^{-\frac{d}{2}} \ dx = \frac{2^{\frac{d}{2}-1}}{\frac{d}{2}-1} M^{1 - \frac{d}{2}}.  
$$
If $M < 2$, 
$$ 
\sum_{t_{l+2} \geq M} n_{l+2}^{-\frac{d}{2}} \leq \zeta(\tfrac{d}{2}) < \zeta(\tfrac{d}{2}) 2^{\frac{d}{2}-1} M^{1-\frac{d}{2}}. 
$$ 
The expression in~\eqref{eq:lemma_3_2} is therefore less than $c M^{1-\frac{d}{2}} n^{-\frac{d}{2}}$.  
Combining this estimate with~\eqref{eq:lemma_3_3} yields 
$$ 
\sum_{\substack{t_1 + \ldots + t_{l+2} = t, \\ t_1, \ldots, t_{l+2} \geq M}} \prod_{j=1}^{l+2} t_j^{-\frac{d}{2}} \leq \frac{c^{l+1}}{M^{(l+1) (\frac{d}{2}-1)}} t^{-\frac{d}{2}}. 
$$
\end{proof}
\end{appendix}

%======================================
%======================================
%\subsection{}
%======================================
%======================================

%======================================
%======================================
%\subsection{}
%======================================
%======================================

%Our proof of Proposition~\ref{prop:key} also uses the following tail estimate for the exponential series. 
%
%\begin{lemma}     \label{lm:tail_exponential}
%Let $f(t)$ be an integer-valued function such that there are $\rho_2 > \rho_1 > 1$ for which $e^{\rho_2 - 1} < \rho_1^{\rho_2}$ and $\rho_1 t < f(t) < \rho_2 t$ for all $t$ sufficiently large. Then, we have for $\lambda > 0$ sufficiently small   
%\begin{equation*}   
%\lim_{t \to \infty} e^{(\lambda -1)t} \sum_{n = f(t)}^{\infty} \frac{t^n}{n!} = 0.
%\end{equation*}
%\end{lemma}

%=======================
%References
%=======================

\bibliographystyle{alpha}

\bibliography{discrete_space_polymer-2}

%======================================
%======================================
%\section{Extra Bits}
%======================================
%======================================

%
%\begin{proposition}     \label{prop:key}
%For $\beta > 0$ sufficiently small, the following statements hold. 
%\begin{enumerate}
%\item[(B1)] There is $\tilde \rho > 0$ such that for every $\rho \in (0,\tilde \rho]$, $c \geq 0$, 
%$$
%\lim_{t \to \infty} e^{(\rho -1) t} \sum_{l \notin J(t)} e^{c t^{\mu} l} \frac{t^l}{l!} \sum_{1 \leq r \leq l+1} r \alpha^r A(t,l,r) = 0. 
%$$
%\item[(B2)] There is $\tilde \rho > 0$ such that for every $\rho \in (0, \tilde \rho]$,
%$$
%\lim_{t \to \infty} e^{(\rho - 1) t} \sum_{l \in J(t)} \frac{t^l}{l!} \sum_{\nuone l \leq r \leq l+1} r \alpha^r A(t,l,r) = 0. 
%$$
%\item[(B3)] We have  
%$$
%\lim_{t \to \infty}  e^{t^{\sublin}} e^{-t} \sum_{l \in J(t) \setminus K(t)}  \frac{t^l}{l!} \sum_{1 \leq r \leq l+1} r \alpha^r A(t,l,r) = 0.
%$$
%\end{enumerate}
%\end{proposition}  
%
%

\end{document}